\documentclass{amsart}

\usepackage{tikz}
\usetikzlibrary{calc}
\usepackage{amsthm}
\usepackage{enumitem}
\usepackage{aliascnt}
\usepackage{float}
\usepackage{hyperref}
\usepackage{multicol}
\definecolor{linkColor}{HTML}{101f7d}
\definecolor{intrColor}{HTML}{BBD58E}
\definecolor{shade1}{HTML}{BBD58E}
\definecolor{shade2}{HTML}{F9A64D}
\definecolor{shade3}{HTML}{FD879B}
\hypersetup{
    colorlinks=true,
    citecolor=linkColor,
    linkcolor=linkColor,
    filecolor=linkColor,      
    urlcolor=linkColor,
    }
\usepackage{csquotes}
\usepackage[
  backend=biber,
  style=alphabetic,
  sorting=nyt,
  maxalphanames=3,
  minalphanames=1,
  doi=false,isbn=false,url=false
]{biblatex}

\usepackage[nameinlink,noabbrev]{cleveref}

\newcommand\C{\mathsf{C}}
\newcommand\T{\mathsf{T}}
\newcommand\SSYT{\mathrm{SSYT}}
\newcommand\Gr{\mathrm{Gr}}
\newcommand\GL{\mathrm{GL}}
\newcommand\Z{\mathbb{Z}}
\newcommand{\iprod}{\mathbin{\lrcorner}}

\theoremstyle{plain}
\newtheorem{theorem}{Theorem}[section]

\newaliascnt{proposition}{theorem}
\newtheorem{proposition}[proposition]{Proposition}
\aliascntresetthe{proposition}

\newaliascnt{lemma}{theorem}
\newtheorem{lemma}[lemma]{Lemma}
\aliascntresetthe{lemma}

\newaliascnt{question}{theorem}

\aliascntresetthe{question}

\newaliascnt{corollary}{theorem}
\newtheorem{corollary}[corollary]{Corollary}
\aliascntresetthe{corollary}

\theoremstyle{definition}

\newaliascnt{definition}{theorem}
\newtheorem{definition}[definition]{Definition}
\aliascntresetthe{definition}

\newaliascnt{example}{theorem}
\newtheorem{example}[example]{Example}
\aliascntresetthe{example}

\newaliascnt{remark}{theorem}
\newtheorem{remark}[remark]{Remark}
\aliascntresetthe{remark}
\crefname{remark}{remark}{remarks}
\Crefname{remark}{Remark}{Remarks}

\newaliascnt{conjecture}{theorem}
\newtheorem{conjecture}[conjecture]{Conjecture}
\aliascntresetthe{conjecture}
\crefname{conjecture}{conjecture}{conjectures}
\Crefname{conjecture}{Conjecture}{Conjectures}

\numberwithin{equation}{section}

\title{Bounded core partitions and Borel--Weil--Bott}
\author{Fern Gossow}
\address{Sydney, Australia}
\email{fernleaf.maths@gmail.com}
\author{Andrew Huchala}
\address{University of Oregon, USA}
\email{ahuchala@uoregon.edu}



\begin{document}

\begin{abstract}
The Borel--Weil--Bott theorem can be used to decompose the cohomology of twisted sheaves of holomorphic forms on the complex Grassmannian into irreducible representations of the general linear group. By analyzing this decomposition, we provide two effective formulae for computing the associated Hodge numbers, and give examples in special cases. One of these involves a novel integer-valued hook-product statistic on bounded partitions, and the other is based on semistandard tableaux. We reformulate Snow's observation that the positivity of the Hodge numbers is equivalent to the existence of a partition satisfying certain properties, and improve known bounds on when this occurs. This involves a combinatorial proof of the Nakano vanishing theorem for the Grassmannian utilizing a map from  core partitions to plane partitions. Finally, we extend our computation of the Hodge numbers of twisted holomorphic forms on the Grassmannian to a q-analogue.
\end{abstract}

\vspace*{-1cm}
\maketitle

\section{Introduction}

The Borel--Weil--Bott (BWB) theorem \cite{Serre95} is a tool for computing the higher sheaf cohomology associated to certain vector bundles, by a decomposition of each cohomology group into irreducible representations of a Lie group. In Type A, this Lie group is the group $\GL_n$ of complex invertible matrices, and its irreducible representations have an underlying combinatorial structure which can be neatly described in terms of partitions and Young tableaux. 

Fix integers $1\leq k< n$ and let $X:=\Gr(k,n)$ denote the Grassmannian of $k$-dimensional subspaces of $\mathbb{C}^n$. Let $\Omega_X^j(t):=\Omega^j_X\otimes\mathcal{O}_X(t)$ be the sheaf of twisted holomorphic $j$-forms on $X$. We are interested in computing the Hodge number
\[h^{j,i}(t):=\dim H^i(\Omega_X^j(t))\]
for integers $j,i\geq 0$ and $t$. The cohomology has a natural action of $\mathrm{GL}_n$ induced from the action on $\mathbb{C}^n$. For a dominant integral weight $\beta$, let $L_\beta$ denote the irreducible representation of highest weight $\beta$ and $L_\beta^\vee$ its dual. In \Cref{sec:BWB} we use the formulation of the BWB theorem from \cite[\S4.1]{Wey03} to observe that
\[H^i(\Omega_X^j(t))\cong\bigoplus_\lambda L_{\beta_\lambda(t)}^\vee\]
where $\beta_\lambda(t)$ is a dominant integral weight associated to $\lambda$ (see \Cref{def:beta}), and the direct sum is over all partitions $\lambda$ such that:
\begin{itemize}
\item $\lambda$ is $(k,n-k)$-bounded (meaning $\lambda$ has at most $k$ rows and $n-k$ columns),
\item $\lambda$ has $j$ boxes,
\item no boxes in $\lambda$ have hook length $t$, and
\item exactly $i$ boxes have hook length exceeding $t$.
\end{itemize}
The positivity of $h^{j,i}(t)$ is then equivalent to the existence of such a partition. This fact was first observed by Snow \cite{Sno86}, and we call $\lambda$ a \emph{Snow partition} with parameters $(k,n,j,t,i)$.

A partition with no boxes of hook length $t$ is known as a $t$-core partition, and the properties and applications of these objects are well-studied in the literature (see \cite{GKS90} for associated generating functions, \cite{GrOn96} for an approach based on representation theory, or \cite{Cho21} for a survey). Partitions which are both $t$-core and bounded in a rectangle were studied in \cite{Ayyer24}, but that analysis does not include the size of the partition or its interior.

A set of five inequalities on the parameters that each imply the nonexistence of Snow partitions was given by Snow, and in \Cref{sect:bounds} we generalize and improve these bounds by mapping each Snow partition to a plane partition using a construction of Chen \cite{Chen10}. In particular, we obtain a combinatorial proof of the Nakano vanishing theorem \cite{Nak54} for $\Omega_X^j(t)$, which is equivalent to every Snow partition with $t>0$ satisfying $i+j\leq N$.

Set $d_\lambda(t)=\dim L_{\beta_\lambda(t)}^\vee$, so $h^{j,i}(t)=\sum_\lambda d_\lambda(t)$ with the sum over all Snow partitions with parameters $(k,n,j,t,i)$. Our main results are two formulae for computing $d_\lambda(t)$ for every Snow partition (previously only special cases were known, such as $t=2$ in \cite{Sno86}). In the following we state the results for $t>0$, but our proofs allow for all integers $t$ by extending some combinatorial definitions (in any case, Serre duality implies that $h^{j,i}(t)=h^{N-j,N-i}(-t)$).

The first formula for $d_\lambda(t)$ is based on a novel integer-valued hook product statistic on bounded partitions (\Cref{def:t-hook-ratio}), which is nonzero if and only if $\lambda$ is $t$-core. Let $\lambda^\C:=(n-k-\lambda_k,\dots,n-k-\lambda_1)$ denote the complement partition to $\lambda$ in its $k\times (n-k)$ bounding rectangle, and $h_\lambda(a,b)$ the hook length of the box in row $a$ and column $b$ of $\lambda$.

\begin{theorem}\label{thm:dimension-formula-intro}
For a Snow partition $\lambda$ with $t>0$,
\[d_\lambda(t)=(-1)^{i+j}\prod_{(a,b)\in\lambda}\frac{h_{\lambda}(a,b)-t}{h_\lambda(a,b)}\prod_{(a,b)\in\lambda^\C}\frac{h_{\lambda^\C}(a,b)+t}{h_{\lambda^\C}(a,b)}.\]
\end{theorem}

The second formula for $d_\lambda(t)$ follows from explicitly calculating $\beta_\lambda(t)$, and observing that $L_{\beta_\lambda(t)}^\vee$ is a polynomial representation of $\GL_n$ when $t>0$ (concretely, all entries of $\beta_\lambda(t)$ are nonpositive), so $d_\lambda(t)$ is enumerated by semistandard Young tableaux.

Given a Snow partition $\lambda$ with $t>0$, let $\partial_\lambda(t)_a$ for $1\leq a\leq k$ be the number of boxes in row $a$ of $\lambda$ with hook length less than $t$.

\begin{theorem}\label{thm:gamma-dimension-intro}
For a Snow partition $\lambda$ with $t>0$, let $\gamma$ be the weakly decreasing sequence of integers with entries given by the multiset union
\[\{\partial_\lambda(t)_a\mid 1\leq a\leq k\}\cup\{t-\partial_{\lambda^\T}(t)_b\mid 1\leq b\leq n-k\}.\]
Then $\gamma$ is an $(n,t)$-bounded partition of size $(n-k)t$, and $d_\lambda(t)$ is equal to the number of semistandard Young tableaux of shape $\gamma$ with entries $1,\dots,n$.
\end{theorem}

We apply these formulae to compute Hodge numbers in specific cases. This involves first finding all Snow partitions $\lambda$ for a given set of parameters, and then calculating $d_\lambda(t)$ for each of these. For example, when $j=0$ the only Snow partition is $\lambda=(0,\dots,0)$ with $i=0$, and comparing the calculation from \Cref{thm:dimension-formula-intro} with a formula of MacMahon \cite{Mac86} gives that $h^{0,0}(t)$ is equal to the number of plane partitions bounded in a $k\times(n-k)\times t$ box.

We also examine the cases $i+j=N$ and $i+j=N-1$, which are extremal cases for Nakano vanishing. These relate closely to ongoing work generalizing that of Fatighenti and Mongardi \cite{Fat21} who construct a Griffiths-type ring for the cohomology of a smooth complex hypersurface $Z\subset X=\mathrm{Gr}(k,n)$, under the assumptions that $H^i(\Omega_X^{N-i}(pZ))=0$ and $H^i(\Omega_X^{N-i-1}(pZ))=0$ for all $i,p>0$. \Cref{thm:i+j=N,thm:i+j=N-1} provide a complete classification of when these vanishings occur. 

Applying Stembridge's notion of $q$-dimension \cite{Stem94} to representations of $\GL_n$, we prove the following $q$-analogue of \Cref{thm:dimension-formula-intro} in \Cref{sec:q}. If
\[[m]_q:=\frac{q^{m/2}-q^{-m/2}}{q^{1/2}-q^{-1/2}}\]
is the (symmetrized) $q$-analogue of an integer $m$, then for every Snow partition $\lambda$ with $t>0$:
\[\dim_q L_{\beta_\lambda(t)}^\vee=(-1)^{i+j}\prod_{(a,b)\in\lambda}\frac{[h_{\lambda}(a,b)-t]_q}{[h_\lambda(a,b)]_q}\prod_{(a,b)\in\lambda^\C}\frac{[h_{\lambda^\C}(a,b)+t]_q}{[h_{\lambda^\C}(a,b)]_q}.\]

In \Cref{sec:euler} we give an algebraic approach for the computation of the Euler characteristic $\chi(\Omega_X^j(t))$, and discuss the combinatorial implications for bounded partitions. \Cref{con:q-euler-characteristic} suggests a $q$-analogue for this formula based on the $q$-binomial coefficient.

\subsection*{Acknowledgments} The authors wish to thank Benjamin Young and Nicolas Addington for their helpful suggestions and supportive comments.

\section{Background}

\subsection{Sequences}
For sequences $a=(a_1,\dots,a_n)\in\mathbb{Z}^n$ and $b=(b_1,\dots,b_m)\in\Z^m$, we let $a\smile b:=(a_1,\dots,a_n,b_1,\dots,b_m)\in\Z^{n+m}$ denote concatenation. For $t\in\mathbb{Z}$, let $a+t:=(a_1+t,\dots,a_n+t)$. If $m=n$, define $a+b:=(a_1+b_1,\dots,a_n+b_n)$. We also define $-a:=(-a_1,\dots,-a_n)$ and $\mathrm{rev}(a):=(a_n,\dots,a_1)$.

We sometimes represent sequences in exponential notation, with the exponent denoting a repeated entry. For example:
\[(1,1,1,1,0,0,5,5,5,0)=(1^4,0^2,5^3,0^1).\]

\subsection{Permutations}
For $n\geq 1$, let $\mathfrak{S}_n$ denote the symmetric group on $n$ letters, whose elements (known as \emph{permutations}) are bijections on $\{1,\dots,n\}$. We express the permutation $\sigma\in\mathfrak{S}_n$ by its \emph{one-line notation} $\sigma(1)\sigma(2)\cdots\sigma(n)$. An \emph{inversion} is a pair $(a,b)$ with $1\leq a<b\leq n$ and $\sigma(a)>\sigma(b)$. This inversion pair is a \emph{descent} (at $\ell$) if $(a,b)=(\ell,\ell+1)$. A permutation is \emph{Grassmannian} if it has at most one descent. The \emph{Coxeter length} of a permutation is equal to its total number of inversions. For more on the Coxeter theory of permutations and its connections to representation theory, see \cite{CoCG}.

We let $\mathfrak{S}_n$ act on the sequence $a\in\mathbb{Z}^n$ by $\sigma(a)=(a_{\sigma(1)},\dots,a_{\sigma(n)})$.

\subsection{Bounded partitions}
\label{sec:partitions}

The primary objects of our combinatorial study are \emph{bounded partitions}, which include their bounding rectangle as part of their combinatorial data. This additional information is essential for some of our definitions.

Fix integers $1\leq k< n$. A $(k,n-k)$-\emph{bounded partition} (or simply, a bounded partition when $k$ and $n$ are understood) is an integer sequence $\lambda=(\lambda_1,\dots,\lambda_k)$ such that
\[n-k\geq\lambda_1\geq\cdots\geq\lambda_k\geq 0.\]
In some cases we extend our definition to allow $n=\infty$, and simply refer to $k$-bounded partitions.

Consider the rectangular grid of boxes $(a,b)$ for $1\leq a\leq k$ and $1\leq b\leq n-k$ specified in matrix coordinates. Let $N:=k(n-k)$ denote the total number of boxes. We say $(a,b)\in\lambda$ if $b\leq \lambda_a$, and draw the \emph{diagram} for $\lambda$ by outlining all such boxes.

\begin{example}
The $(3,8)$-bounded partition $\lambda=(4,1,0)$ has diagram:
\[\begin{tikzpicture}[scale=0.6]
\draw [black!50!white] (0,0) grid (5,-3);
\draw [thick] (0,0) grid (4,-1);
\draw [thick] (0,-1) grid (1,-2);
\end{tikzpicture}\]
\end{example}

The \emph{transpose} of $\lambda$ is the $(n-k,k)$-bounded partition $\lambda^\T$ whose diagram is obtained by reflecting the diagram for $\lambda$ along $(a,b)\mapsto (b,a)$. The \emph{complement} $\lambda^\C$ is the $(k,n-k)$-bounded partition obtained by rotating the bounding box of $\lambda$ by $180^\circ$ and taking the set-theoretic complement of $\lambda$. That is,
\[\lambda^\C=(n-k-\lambda_k,\dots,n-k-\lambda_1).\]
For $1\leq a\leq k$ and $1\leq b\leq n-k$, define the \emph{hook length} of the box $(a,b)$ by
\[h_\lambda(a,b):=\lambda_a+\lambda_b^\T-a-b+1.\]
We make special note that this definition applies to boxes both inside and outside of $\lambda$. When $(a,b)\in\lambda$, $h_\lambda(a,b)$ is the length of the south-east hook originating from the box $(a,b)$ to the boundary of $\lambda$. Otherwise, the hook length is negative, and its magnitude is the associated hook length in $\lambda^\C$.

\begin{example}
For the $(3,8)$-bounded partition $\lambda=(4,1,0)$ we notate every box by its hook length.
\[\begin{tikzpicture}[scale=0.6]
\draw [black!50!white] (0,0) grid (5,-3);
\draw [thick] (0,0) grid (4,-1);
\draw [thick] (0,-1) grid (1,-2);

\node at (0.5,-0.5) {$5$};
\node at (1.5,-0.5) {$3$};
\node at (2.5,-0.5) {$2$};
\node at (3.5,-0.5) {$1$};
\node at (4.5,-0.5) {$-1$};

\node at (0.5,-1.5) {$1$};
\node at (1.5,-1.5) {$-1$};
\node at (2.5,-1.5) {$-2$};
\node at (3.5,-1.5) {$-3$};
\node at (4.5,-1.5) {$-5$};

\node at (0.5,-2.5) {$-1$};
\node at (1.5,-2.5) {$-3$};
\node at (2.5,-2.5) {$-4$};
\node at (3.5,-2.5) {$-5$};
\node at (4.5,-2.5) {$-7$};
\end{tikzpicture}\]
\end{example}

For \emph{any} integer $t$, a bounded partition $\lambda$ is called \emph{$t$-core} if it has no boxes of hook length $t$. Note that $\lambda$ is $(-t)$-core if and only if $\lambda^\C$ is $t$-core. By a standard argument using the abacus construction (see \cite[Theorem 2.7.40]{Gor81}), if a partition is $t$-core then it is also $(mt)$-core for every $m\geq 1$.

Define the $t$-\emph{interior} of $\lambda$ to be the $(k,n-k)$-bounded partition made up of all boxes $(a,b)$ such that $h_\lambda(a,b)\geq t$. Also let $|\lambda|:=\lambda_1+\cdots+\lambda_k$ denote the \emph{size} of a partition, which counts its boxes. We now define the objects which index the summands of the decomposition of $H^i(\Omega_X^j(t))$ into irreducible representations of $\GL_n$, as mentioned in the introduction.

\begin{definition}
Fix a tuple of integers $(k,n,j,t,i)$ such that $1\leq k<n$ and $j,i\geq 0$. A \emph{Snow partition} with \emph{parameters} $(k,n,j,t,i)$ is a $(k,n-k)$-bounded partition with $|\lambda|=j$, which is $t$-core with $|\iota_\lambda(t)|=i$.
\end{definition}

We are interested in a classification of the parameters $(k,n,j,t,i)$ for which such partitions exist, as these are the parameters for which $h^{j,i}(t)>0$ (see \Cref{thm:cohomology-iso}). It is clear that $0\leq i,j\leq N$, so $h^{j,i}(t)=0$ if $j>N$ or $i>N$. If $t\geq 0$, then $\iota_\lambda(t)\subseteq \lambda$, so $i\leq j$. If $t\leq 0$, then $\iota_\lambda(t)\supseteq\lambda$ so $i\geq j$. Since the only $1$-core partition is empty, if $t=1$ we must have $i=j=0$.

\begin{definition}\label{def:delta}
For a Snow partition $\lambda$ with $t\geq 0$, set $\partial_\lambda(t):=\lambda-\iota_\lambda(t)$.
\end{definition}

It turns out that $\partial_\lambda(t)$ is always a partition. It is obtained by taking the boxes in the $t$-\emph{boundary} skew partition $\lambda/\iota_\lambda(t)$ and sliding them to the left. Moreover, the map $\lambda\mapsto\partial_\lambda(t)$ can be inverted.

\begin{lemma}[{\cite[Theorem 7]{Lap04}}]\label{lem:bounded-bijection}
For $t>0$, the map $\lambda\mapsto\partial_\lambda(t)$ is a bijection between $k$-bounded $t$-core partitions and $(k,t-1)$-bounded partitions.
\end{lemma}

\begin{figure}[H]
\[\begin{tikzpicture}[scale = 0.6]
\draw [black!50!white] (0,0) grid (6,-5);
\fill [intrColor] (0,0) -- (3,0) -- (3,-1) -- (1,-1) -- (1,-3) -- (0,-3) -- cycle;
\draw [thick] (0,0) grid (6,-1);
\draw [thick] (0,-1) grid (3,-3);
\draw [thick] (0,-3) grid (1,-5);
\node [above] at (3,0) {$\lambda$};

\draw [black!50!white] (-4,0) grid (-1,-5);
\draw [thick] (-4,0) grid (-1,-1);
\draw [thick] (-4,-1) grid (-2,-3);
\draw [thick] (-4,-3) grid (-3,-5);
\node [above] at (-2.5,0) {$\partial_\lambda(4)$};
\end{tikzpicture}\]
\caption{A Snow partition $\lambda=(6,3,3,1,1)$ for the parameters $k=5$, $n=11$, $j=14$, $t=4$ and $i=5$, with its $4$-interior $\iota_\lambda(4)=(3,1,1,0,0)$ shaded. The $(5,3)$-bounded partition $\partial_\lambda(4)$ is on the left.}
\label{fig:partial}
\end{figure}

\subsection{Tableaux}
\label{sec:tableaux}

For a partition $\lambda$, a \emph{tableau} of shape $\lambda$ is a labeling of the boxes of $\lambda$ by positive integers. The size $|T|$ of a tableau $T$ is the sum of its labels. We give some basic enumerations of classes of tableaux below, and refer to \cite{Stan99} for the general theory.

Fix a partition $\lambda$ of size $j$. For an integer $m\geq 1$, a \emph{semistandard Young tableau} is a labeling of the boxes of $\lambda$ by the integers $1,\dots,m$ such that the labels increase weakly along rows and strictly down columns. The number of such semistandard Young tableaux can be counted as
\begin{equation}\label{eq:SSYT}\#\SSYT(\lambda;m)=\prod_{1\leq p<q\leq m}\frac{\lambda_p-\lambda_q+q-p}{q-p}=\prod_{(a,b)\in\lambda}\frac{m+b-a}{h_\lambda(a,b)}.\end{equation}

The \emph{content} of a semistandard tableau is the sequence $\mu=(\mu_1,\dots,\mu_m)$ where $\mu_\ell$ is the number of boxes with label $\ell$. The generating function for semistandard Young tableaux weighted by content is the \emph{Schur polynomial}
\[s_\lambda(x_1,\dots,x_m)=\sum_{\mu\in\mathbb{Z}^m}\#\SSYT(\lambda;\mu)\,x_1^{\mu_1}\cdots x_m^{\mu_m}\]
where $\SSYT(\lambda;\mu)$ is the set of semistandard Young tableaux of shape $\lambda$ and content $\mu$. Schur polynomials are symmetric in the variables $x_1,\dots,x_m$.

\begin{example}
If $\lambda=(2,2)$, the elements of $\SSYT(\lambda;3)$ are
\[\begin{tikzpicture}[scale = 0.4]
\draw (0,0) grid (2,-2);
\node at (0.5,-0.5) {$1$};
\node at (1.5,-0.5) {$1$};
\node at (0.5,-1.5) {$2$};
\node at (1.5,-1.5) {$2$};

\begin{scope}[shift = {(3,0)}]
\draw (0,0) grid (2,-2);
\node at (0.5,-0.5) {$1$};
\node at (1.5,-0.5) {$1$};
\node at (0.5,-1.5) {$2$};
\node at (1.5,-1.5) {$3$};
\end{scope}

\begin{scope}[shift = {(6,0)}]
\draw (0,0) grid (2,-2);
\node at (0.5,-0.5) {$1$};
\node at (1.5,-0.5) {$1$};
\node at (0.5,-1.5) {$3$};
\node at (1.5,-1.5) {$3$};
\end{scope}

\begin{scope}[shift = {(9,0)}]
\draw (0,0) grid (2,-2);
\node at (0.5,-0.5) {$1$};
\node at (1.5,-0.5) {$2$};
\node at (0.5,-1.5) {$2$};
\node at (1.5,-1.5) {$3$};
\end{scope}

\begin{scope}[shift = {(12,0)}]
\draw (0,0) grid (2,-2);
\node at (0.5,-0.5) {$1$};
\node at (1.5,-0.5) {$2$};
\node at (0.5,-1.5) {$3$};
\node at (1.5,-1.5) {$3$};
\end{scope}

\begin{scope}[shift = {(15,0)}]
\draw (0,0) grid (2,-2);
\node at (0.5,-0.5) {$2$};
\node at (1.5,-0.5) {$2$};
\node at (0.5,-1.5) {$3$};
\node at (1.5,-1.5) {$3$};
\end{scope}
\end{tikzpicture}\]
so $s_\lambda(x_1,x_2,x_3)=x_1^2x_2^2+x_1^2x_2x_3+x_1^2x_3^2+x_1 x_2^2x_3+x_1x_2x_3^2+x_2^2x_3^2$.
\end{example}

A \emph{plane partition} is a tableau where the labels weakly decrease both along rows and down columns. We call a plane partition $(k_1,k_2,k_3)$-bounded if its shape $\lambda$ is $(k_1,k_2)$-bounded and every label is at most $k_3$. MacMahon in \cite{Mac86} proved that the set of such plane partitions is enumerated by
\[\#\mathrm{PP}(k_1,k_2,k_3)=\prod_{a=1}^{k_1}\prod_{b=1}^{k_2}\frac{a+b+k_3-1}{a+b-1}.\]

\subsection{Representation theory}\label{sec:rep-theory}

We recall some of the standard Lie theory in Type A, including a classification of (finite dimensional) rational representations for the matrix Lie group $\mathrm{GL}_n$ of invertible matrices over $\mathbb{C}$. The general theory can be found in \cite{Hum72} and \cite{Ful91}.

Given a finite dimensional complex vector space $V$ of dimension $m$ and an action of $\GL_n$ on $V$, we can construct an explicit homomorphism $\GL_n\to\GL_m$ by choosing a basis for $V$. We call the representation $V$ \emph{rational} (resp. \emph{polynomial}) if this map is a rational function (resp. polynomial) in terms of the entries of the input matrix. This classification is independent of the choice of basis on $V$.

The \emph{weight lattice} of $\GL_n$ is identified with $\Z^n$. We let $e_p$ be the $p^\text{th}$ standard basis vector and $\langle-,-\rangle$ the usual inner product. Every rational representation $V$ can be decomposed into \emph{weight spaces} as $V=\bigoplus_{\mu\in\Z^n}V[\mu]$, where $V[\mu]$ is the subspace of elements $v\in V$ such that every diagonal matrix acts by
\[\mathrm{diag}(x_1,\dots,x_n)\cdot v=x_1^{\mu_1}\cdots x_n^{\mu_n}v.\]
The \emph{Weyl group} of $\GL_n$ is $\mathfrak{S}_n$. It acts by linear isomorphism on the weight spaces, so $\dim V[\sigma(\mu)]=\dim V[\mu]$ for every $\mu\in\Z^n$ and $\sigma\in\mathfrak{S}_n$. The dual representation $V^\vee$ satisfies $\dim V^\vee[\mu]=\dim V[-\mu]$ for every $\mu\in\mathbb{Z}^n$ (see \cite[\S0.2]{Kop84} or \cite[Theorem IV.5.A]{Weyl39}). 

The \emph{positive roots} of $\GL_n$ are $\Phi^+:=\{e_r-e_s\mid 1\leq r<s\leq n\}$ and we define a partial order on $\mathbb{Z}^n$ by $\mu\preceq\lambda$ if $\lambda-\mu$ is a sum of positive roots. The \emph{Weyl vector} is $\varrho=\frac{1}{2}\sum_{\phi\in\Phi^+}\phi$, or explicitly,
\[\varrho=\left(\frac{n-1}{2},\frac{n-3}{2},\dots,\frac{-n+1}{2}\right).\]
Every finite dimensional rational representation of $\GL_n$ is a direct sum of irreducible representations. The isomorphism classes of irreducible representations of $\GL_n$ are indexed by the set of \emph{dominant integral weights}
\[\Lambda_n:=\{\beta\in\Z^n\mid \beta_1\geq\beta_2\geq\cdots\geq\beta_n\}.\]
The irreducible representation $L_\beta$ associated to $\beta$ is a \emph{highest-weight representation} for $\beta$, meaning $\dim L_\beta[\beta]=1$ and $\dim L_\beta[\mu]=0$ unless $\mu\preceq\beta$. We provide an explicit construction of these representations in \Cref{sec:schur}. Below are some common examples:
\begin{center}
\begin{tabular}{l l}
$\beta$ & $L_\beta$ \\
\hline
$(0^n)$ & trivial representation\\
$(1^n)$ & determinant representation\\
$(1^m, \,0^{n-m})$ & ${\bigwedge}^m\mathbb{C}^n$ \\
$(m,\,0^{n-1})$ & $\mathrm{Sym}^m\mathbb{C}^n$
\end{tabular}
\end{center}
The dimension of $L_\beta$ is given by the Weyl dimension formula:
\[\dim L_\beta=\prod_{\phi\in\Phi^+}\frac{\langle \phi,\beta+\varrho\rangle}{\langle\phi,\varrho\rangle}=
\prod_{1\leq r<s\leq n}\frac{\beta_r-\beta_s+s-r}{s-r}.\]
\begin{remark}
    For most of the paper we will replace $\varrho$ by the \emph{shifted Weyl vector} $\varrho'=(n,\dots,2,1)$ for ease in calculations. It is often the case that the choice between $\varrho$ and $\varrho'$ is inconsequential, such as when using the Weyl dimension formula.
\end{remark}
The representation $L_\beta$ is polynomial if and only if $\beta$ belongs to the set
\[\Lambda_n^+:=\{\lambda\in\mathbb{Z}^n\mid \lambda_1\geq\cdots\geq\lambda_n\geq 0\},\]
which we identify with the set of $n$-bounded partitions. For a weight $\lambda\in\Lambda_{n}^+$ we have $\dim L_\lambda=\#\mathrm{SSYT}(\lambda;n)$, and more precisely, $\dim L_\lambda[\mu]=\#\SSYT(\lambda;\mu)$ for every weight $\mu\in\mathbb{Z}^n$.

\begin{remark}
Because $\lambda$ is a \emph{bounded partition}, there is no ambiguity about which group is acting on $L_\lambda$ (or $L_\beta$). For example, $L_{(2,1)}$ is a representation of $\GL_2$, and $L_{(2,1,0,0)}$ is a representation of $\GL_4$.
\end{remark}

We state some relationships amongst the irreducible representations $L_\beta$ which are relevant in this work. The dual of $L_\beta$ is isomorphic to $L_{-\mathrm{rev}(\beta)}$, and
\begin{equation}\label{eq:tensor-shift}
L_\beta\otimes L_{(t^n)}\cong L_{\beta+t}
\end{equation}
for any $\beta\in\Lambda_n$ and $t\in\mathbb{Z}$ as a special case of the Littlewood--Richardson rule. Finally, we have the Cauchy decomposition \cite[Exercise 6.11]{Ful91} which gives the isomorphism of representations of $\GL_k\times\GL_{n-k}$:
\begin{equation}\label{eq:cauchy decomposition}
{\bigwedge}^j\left(\mathbb{C}^k\otimes \mathbb{C}^{n-k}\right)\cong \bigoplus_{|\lambda|=j} L_\lambda\otimes L_{\lambda^\T}
\end{equation}
where $\lambda$ runs over all $(k,n-k)$-bounded partitions of size $j$.

\subsection{Schur Functors}\label{sec:schur}
We give an explicit construction of the irreducible $\GL_n$ representation $L_\beta$ for any $\beta\in\Lambda_n$. Importantly, this construction is \emph{functorial}, and will be applied to vector bundles in the setup of the Borel--Weil--Bott theorem.

For every $\beta\in\Lambda_n$, we construct a functor $\Sigma^\beta$ from the category of rational representations of $\mathrm{GL}_n$ to itself, such that $\Sigma^\beta(\mathbb{C}^n)\cong L_\beta$ for every $\beta\in\Lambda_n$, where $\mathbb{C}^n$ is equipped with the standard action. For a more detailed version of this construction, see \cite[\S6]{Ful91}.

We start with $\lambda\in\Lambda_n^+$, so $\lambda$ is an $n$-bounded partition of some size $j$. Consider the standard tableau $T_\lambda$ obtained by labeling the boxes of $\lambda$ by $1,\dots,j$ row-by-row from left to right, then top to bottom. Define subgroups of $\mathfrak{S}_j$ by:
\begin{align*}
P_\lambda&:=\{\sigma\in\mathfrak{S}_j\mid \text{$\sigma$ preserves the rows of $T_\lambda$}\} \\
Q_\lambda&:=\{\sigma\in\mathfrak{S}_j\mid \text{$\sigma$ preserves the columns of $T_\lambda$}\}
\end{align*}
The group algebra $\mathbb{C}[\mathfrak{S}_j]$ consists of formal $\mathbb{C}$-linear combinations of the elements of $\mathfrak{S}_j$. Let $\ell(\sigma)$ be the Coxeter length of each permutation $\sigma\in\mathfrak{S}_j$, and define the \emph{Young symmetrizer}
\[c_\lambda=\sum_{\sigma\in P_\lambda}\sigma\sum_{\tau\in Q_\lambda}(-1)^{\ell(\tau)}\tau\in\mathbb{C}[\mathfrak{S}_j].\]
Let $V$ be a rational representation of $\GL_n$. The space $V^{\otimes j}:=V\otimes\cdots\otimes V$ has a left action of $\GL_n$ (acting on each factor simultaneously) and a right action of $\mathfrak{S}_j$ (permuting the tensor factors). These actions commute, and we define $\Sigma^\lambda(V)$ as the image of the action of $c_\lambda$ on $V^{\otimes j}$. To obtain all dominant integral weights $\beta\in\Lambda_n$ we use the definition
\[\Sigma^\beta(V):=\Sigma^{\beta+r}(V)\otimes \Sigma^{(r^n)}(V^\vee)\]
for $r\geq 0$. This is well-defined, and we can use this to ``translate'' any dominant integral weight to one with nonnegative values by making $r$ sufficiently large.

We state a number of facts about this construction, which mirror isomorphisms between the irreducible representations $L_\beta$. Applying $\Sigma^\beta$ to the dual of $V$ gives $\Sigma^\beta(V^\vee)\cong\Sigma^\beta(V)^\vee\cong\Sigma^{-\mathrm{rev}(\beta)}(V)$, and $\Sigma^{\beta+t}(V)\cong\Sigma^\beta(V)\otimes\Sigma^{(t^n)}(V)$. We also have the Cauchy decomposition (see (2.3.3) in \cite{Wey03}): for $V$ and $W$ representations of $\GL_k$ and $\GL_{n-k}$, respectively,
\begin{equation}\label{eq:schur cauchy decomposition}
{\bigwedge}^j\left(V\otimes W\right)\cong \bigoplus_{|\lambda|=j} \Sigma^\lambda(V)\otimes \Sigma^{\lambda^\T}(W)
\end{equation}
as representations of $\GL_k\times \GL_{n-k}$, where the sum is over all $(k,n-k)$-bounded partitions $\lambda$ of size $j$. Finally, if $V=\mathbb{C}^n$ is the standard representation of $\GL_n$, then $\Sigma^\beta(V)\cong L_\beta$ \cite{Kop84}.

This construction can be extended to vector bundles (see \cite{Wey03}, and note that their notation $L_{\lambda}E$ corresponds to our $\Sigma^{\lambda^\T}(E)$).

\section{Computation of the Hodge numbers}\label{sec:hodge-formula}

\subsection{BWB for the Grassmannian}\label{sec:BWB}

The Borel--Weil--Bott (BWB) theorem allows us to decompose the cohomology of twisted holomorphic $j$-forms on the Grassmannian into irreducible representations of $\GL_n$. Our construction follows \cite[\S4.1]{Wey03}; an approachable account of this is available in \cite[Theorem A.9]{Nik25}.

Fix $V=\mathbb C^n$, equipped with the standard action of $\mathrm{GL}_n$. The Grassmannian $\mathrm{Gr}(k,V)$ of $k$-dimensional subspaces of $V$ is denoted by $X=\mathrm{Gr}(k,n)$. The structure sheaf of $X$ is denoted $\mathcal O_X$, and the sheaf of holomorphic $j$-forms on $X$ is denoted $\Omega_X^j$. For an overview of complex geometry in this setting, see \cite[\S10.1.1]{Voi01}. Consider the tautological sequence 
\begin{equation*}
0\to S\to \mathcal O_X\otimes V\to Q\to 0
\end{equation*}
for $S$ and $Q$ the tautological subbundle and quotient bundle, respectively. The dual of a bundle $E$ is denoted $E^\vee$. Recall that $\Omega_X\cong S\otimes Q^\vee$. There is an isomorphism $\mathcal O_X(1)\cong\bigwedge^{n-k} Q= \det Q$, and we define $\mathcal O_X(d):=\mathcal O_X(1)^{\otimes d}$. Let $\Omega_X^j := \bigwedge^j \Omega_X
$ and $\Omega_X^j(t):= \Omega_X^j\otimes \mathcal O_X(t)$. Applying \Cref{eq:schur cauchy decomposition} to the twisted cotangent bundle $\Omega_X^j(t)$ gives
\begin{align}\label{thm:wedge of product omega}\Omega^{j}_{X}(t)&\cong \bigoplus_{\lambda} \Sigma^{\lambda} S\otimes \Sigma^{\lambda^\T}(Q^\vee)\otimes \Sigma^{(t^{n-k})}(Q)\\
    &\cong \bigoplus_{\lambda} \Sigma^{\lambda} S\otimes \Sigma^{\lambda^\T-t} (Q^\vee)\notag\\
    &\cong \bigoplus_{\lambda} \Sigma^{-\mathrm{rev}(\lambda) } (S^\vee)\otimes \Sigma^{\lambda^\T-t} (Q^\vee)\notag
\end{align}
with the sum over all $(k,n-k)$-bounded partitions $\lambda$ of size $j$. Under this isomorphism $\Omega_X^j(t)$ is a representation of $\GL_k\times\GL_{n-k}$. The weight of such a representation can be described by a \emph{Levi dominant weight} $\alpha=\delta\smile\epsilon\in\Z^n$ where $\delta\in\Lambda_k$ and $\epsilon\in\Lambda_{n-k}$. In particular, $\alpha$ satisfies the property that $\alpha_i\geq\alpha_{i+1}$ for all $i\neq k$.

We now state the Borel--Weil--Bott theorem for the Grassmannian. Recall that $V=\mathbb{C}^n$ has the standard $\GL_n$ action and $\varrho'=(n,\dots,1)$.
\begin{theorem}[Borel--Weil--Bott, adapted from {\cite[Corollary 4.1.9]{Wey03}}]\label{thm:bwb}
    Set $X=\Gr(k,n)$ and suppose $\alpha=\delta\smile\epsilon\in\mathbb{Z}^n$ is a Levi dominant weight, where $\delta\in\Lambda_k$ and $\epsilon\in\Lambda_{n-k}$. Let $\mathcal V(\alpha)$ denote the vector bundle
    \[\mathcal V(\alpha) = \Sigma^{\delta}(S^\vee) \otimes \Sigma^{\epsilon} (Q^\vee).
    \]
    Then one of two mutually exclusive possibilities occurs:
    \begin{enumerate}
        \item $\alpha + \varrho'$ has repeated entries, and the cohomology of $\mathcal V(\alpha)$ vanishes; or
        \item let $\sigma\in\mathfrak{S}_n$ be the (unique) permutation which orders $\alpha+\varrho'$ into strictly decreasing order, and $\ell$ its Coxeter length. Set $\beta:=\sigma(\alpha+\varrho')-\varrho'\in\Lambda_n$. Then the cohomology of $\mathcal V(\alpha)$ is zero in all degrees except $\ell$, and
\[H^\ell(X,\mathcal{V}(\alpha))\cong\Sigma^{\beta}(V^\vee)\cong L_{\beta}^\vee.\]
    \end{enumerate}
\end{theorem}

\begin{remark}
    The notation used in \cite{Wey03} describes $H^\ell(X,\mathcal V(\alpha))$ in terms of the (skew) Weyl functor $K_{\beta}$ rather than the (skew) Schur functor; these two functors are isomorphic over characteristic zero \cite[Proposition 2.1.18]{Wey03}.
\end{remark}

We will adapt this theorem for $\Omega_X^j(t)$ through \Cref{thm:wedge of product omega}. We first define the Levi dominant weight $\alpha$ which corresponds to each partition $\lambda$ indexing the direct sum.

\begin{definition}\label{def:alpha}
For a $(k,n-k)$-bounded partition $\lambda$ and an integer $t$, set
\[\alpha_\lambda(t):=-\mathrm{rev}(\lambda)\smile(\lambda^\T-t)\in\mathbb{Z}^n.\]
\end{definition}

\begin{lemma}\label{lem:unique-entries}
The sequence $\alpha_\lambda(t)+\varrho'$ has unique entries if and only if $\lambda$ is $t$-core. In this case, the permutation $\sigma$ which reorders $\alpha_\lambda(t)+\varrho'$ in descending order has Coxeter length $\ell=|\iota_\lambda(t)|$.
\end{lemma}
\begin{proof}
Since $-\mathrm{rev}(\lambda)$ and $\lambda^\T-t$ are both weakly decreasing and $\varrho'$ is strictly decreasing, we only need to compare between the first $k$ and last $n-k$ terms of $\alpha+\varrho'$. For $1\leq a\leq k$ and $1\leq b\leq n-k$, we have
\[(\alpha+\varrho')_{k-a+1}=-\lambda_a+n-k+a\text{ and }(\alpha+\varrho')_{k+b}=\lambda_b^\T-t+n-k-b+1.\]
By rearranging, these entries are equal if and only if
\[\lambda_a+\lambda_b^\T-a-b+1=t,\]
so $\lambda$ has a hook length of $t$. The Coxeter length of $\sigma$ is the number of pairs of entries in $\alpha_\lambda(t)+\varrho'$ which are already in ascending order. This happens precisely when $(\alpha+\varrho')_{k-a+1}<(\alpha+\varrho')_{k+b}$, which is equivalent to $h_{\lambda}(a,b)>t$ by the above calculation.
\end{proof}

\begin{definition}\label{def:beta}
For a bounded $t$-core partition $\lambda$, define the \emph{BWB weight}
\[\beta_\lambda(t):=\sigma(\alpha_\lambda(t)+\varrho')-\varrho'\in\Lambda_n\]
and set $d_\lambda(t):=\dim L_{\beta_\lambda(t)}^\vee$. 
\end{definition}
We now simplify the statement of \Cref{thm:bwb}. This directly implies Snow's combinatorial condition on the positivity of Hodge numbers \cite[\S3.1]{Sno86}.

\begin{theorem}\label{thm:cohomology-iso}
Fix integers $(k,n,j,t,i)$ such that $1\leq k< n$ and $j,i\geq 0$. Then
\[H^i(\Gr(k,n),\Omega^j(t))\cong\bigoplus_\lambda L_{\beta_\lambda(t)}^\vee\]
as representations of $\GL_n$, with the direct sum over all Snow partitions $\lambda$ with parameters $(k,n,j,t,i)$. In particular, $h^{j,i}(t)=\sum_\lambda d_\lambda(t)$.
\end{theorem}

\subsection{The BWB weight}

Fix a Snow partition $\lambda$ with parameters $(k,n,j,t,i)$. In this section we explicitly determine the permutation $\sigma$ which reorders $\alpha_\lambda(t)+\varrho'$, and hence the BWB weight $\beta_\lambda(t)\in\Lambda_n$. This gives a formula for $d_\lambda(t)$ based on semistandard Young tableaux. In the next section we use a different approach to obtain the explicit formula for $d_\lambda(t)$ from \Cref{thm:dimension-formula-intro}.

\begin{lemma}\label{lem:sigma}
For a bounded $t$-core partition $\lambda$, the permutation $\sigma$ which rearranges $\alpha_\lambda(t)+\varrho'$ into descending order is given by
\[\sigma(k-a+1)=k-a+1+\iota_\lambda(t)_a\text{ and }\sigma(k+b)=k+b-\iota_\lambda(t)_b^\T\]
for every $1\leq a\leq k$ and $1\leq b\leq n-k$.
\end{lemma}

\begin{proof}
We consider $\sigma$ as a sequence of adjacent transpositions which swap ascending pairs, until the sequence is weakly decreasing. 

The entry $(\alpha+\varrho')_{k}$ swaps with every entry $(\alpha+\varrho')_{k+b}$ such that $(\alpha+\varrho')_{k}<(\alpha+\varrho')_{k+b}$, or equivalently, $h_\lambda(1,b)>t$ by the proof of \Cref{lem:unique-entries}, so it moves to the right $\iota_\lambda(t)_1$ times. Next, $(\alpha+\varrho')_{k-1}$ swaps with every entry $(\alpha+\varrho')_{k+b}$ such that $h_\lambda(2,b)>t$, so it moves to the right $\iota_\lambda(t)_2$ times, and so on.

Alternatively, we can first swap $(\alpha+\varrho')_{k+1}$ with every $(\alpha+\varrho')_{k-a+1}$ such that $(\alpha+\varrho')_{k+1}>(\alpha+\varrho')_{k-a+1}$, or equivalently $h_\lambda(a,1)>t$, which moves it to the left $\iota_\lambda(t)_1^\T$ times. Then $(\alpha+\varrho')_{k+2}$ moves to the left $\iota_\lambda(t)_2^\T$ times and so on.
\end{proof}

\begin{example}
Take the $(4,9)$-bounded partition $\lambda=(5,2,2,1)$, which we consider as both a $0$-core and $5$-core partition. We have $\alpha_\lambda(0)+\varrho'=(8,6,5,1,9,7,4,3,2)$ with $\sigma=245913678$, which records that $8$ is the second largest entry of $\alpha_\lambda(0)+\varrho'$, $6$ is the fourth largest, and so on. Similarly, $\alpha_\lambda(5)+\varrho'=(8,6,5,1,4,2,-1,-2,-3)$ with $\sigma=123645789$.

In \Cref{fig:permutation} we draw the crossing diagram for both permutations. Drawing a rotated box around each crossing recovers the shapes $\iota_\lambda(0)=\lambda$ and $\iota_\lambda(5)$.
\begin{figure}[ht]
\begin{center}
\begin{tikzpicture}[scale = 0.6]
\draw [thick, black, fill=intrColor] (0.5,-3.5) -- (1.5,-2.5) -- (2.5,-3.5) -- (1.5,-4.5) -- cycle;
\draw [thick, black, fill=intrColor] (1.5,-2.5) -- (2.5,-1.5) -- (3.5,-2.5) -- (2.5,-3.5) -- cycle;
\draw [thick, black, fill=intrColor] (2.5,-1.5) -- (3.5,-0.5) -- (4.5,-1.5) -- (3.5,-2.5) -- cycle;
\draw [thick, black, fill=intrColor] (3.5,-0.5) -- (4.5,0.5) -- (5.5,-0.5) -- (4.5,-1.5) -- cycle;
\draw [thick, black, fill=intrColor] (1.5,-4.5) -- (2.5,-3.5) -- (3.5,-4.5) -- (2.5,-5.5) -- cycle;
\draw [thick, black, fill=intrColor] (2.5,-5.5) -- (3.5,-4.5) -- (4.5,-5.5) -- (3.5,-6.5) -- cycle;
\draw [thick, black, fill=intrColor] (3.5,-6.5) -- (4.5,-5.5) -- (5.5,-6.5) -- (4.5,-7.5) -- cycle;
\draw [thick, black, fill=intrColor] (4.5,-7.5) -- (5.5,-8.5) -- (6.5,-7.5) -- (5.5,-6.5) -- cycle;
\draw [thick, black, fill=intrColor] (2.5,-3.5) -- (3.5,-2.5) -- (4.5,-3.5) -- (3.5,-4.5) -- cycle;
\draw [thick, black, fill=intrColor] (3.5,-2.5) -- (4.5,-1.5) -- (5.5,-2.5) -- (4.5,-3.5) -- cycle;

\draw [thick, black, dashed] (0,0)  -- (4,0)  -- (5,-1) -- (7,-1);
\draw [thick, black, dashed] (0,-1) -- (3,-1) -- (5,-3) -- (7,-3);
\draw [thick, black, dashed] (0,-2) -- (2,-2) -- (4,-4) -- (7,-4);
\draw [thick, black, dashed] (0,-3) -- (1,-3) -- (6,-8) -- (7,-8);
\draw [thick, black, dashed] (0,-4) -- (1,-4) -- (5,0)  -- (7,0);
\draw [thick, black, dashed] (0,-5) -- (2,-5) -- (5,-2) -- (7,-2);
\draw [thick, black, dashed] (0,-6) -- (3,-6) -- (4,-5) -- (7,-5);
\draw [thick, black, dashed] (0,-7) -- (4,-7) -- (5,-6) -- (7,-6);
\draw [thick, black, dashed] (0,-8) -- (5,-8) -- (6,-7) -- (7,-7);
\foreach \i in {1,2,...,8,9} {
	\node at (-0.3,-\i+1) {$\i$};
	\node at (7.3,-\i+1) {$\i$};
}
\node at (3.5,1) {$t=0$};

\begin{scope}[shift = {(10,0)}]
\draw [thick, black, fill=intrColor] (0.5,-3.5) -- (1.5,-2.5) -- (2.5,-3.5) -- (1.5,-4.5) -- cycle;
\draw [thick, black] (1.5,-2.5) -- (2.5,-1.5) -- (3.5,-2.5) -- (2.5,-3.5) -- cycle;
\draw [thick, black] (2.5,-1.5) -- (3.5,-0.5) -- (4.5,-1.5) -- (3.5,-2.5) -- cycle;
\draw [thick, black] (3.5,-0.5) -- (4.5,0.5) -- (5.5,-0.5) -- (4.5,-1.5) -- cycle;
\draw [thick, black, fill=intrColor] (1.5,-4.5) -- (2.5,-3.5) -- (3.5,-4.5) -- (2.5,-5.5) -- cycle;
\draw [thick, black] (2.5,-5.5) -- (3.5,-4.5) -- (4.5,-5.5) -- (3.5,-6.5) -- cycle;
\draw [thick, black] (3.5,-6.5) -- (4.5,-5.5) -- (5.5,-6.5) -- (4.5,-7.5) -- cycle;
\draw [thick, black] (4.5,-7.5) -- (5.5,-8.5) -- (6.5,-7.5) -- (5.5,-6.5) -- cycle;
\draw [thick, black] (2.5,-3.5) -- (3.5,-2.5) -- (4.5,-3.5) -- (3.5,-4.5) -- cycle;
\draw [thick, black] (3.5,-2.5) -- (4.5,-1.5) -- (5.5,-2.5) -- (4.5,-3.5) -- cycle;

\draw [thick, black, dashed] (0,0)  -- (7,0);
\draw [thick, black, dashed] (0,-1) -- (7,-1);
\draw [thick, black, dashed] (0,-2) -- (7,-2);
\draw [thick, black, dashed] (0,-3) -- (1,-3) -- (3,-5) -- (7,-5);
\draw [thick, black, dashed] (0,-4) -- (1,-4) -- (2,-3) -- (7,-3);
\draw [thick, black, dashed] (0,-5) -- (2,-5) -- (3,-4)  -- (7,-4);
\draw [thick, black, dashed] (0,-6) -- (7,-6);
\draw [thick, black, dashed] (0,-7) -- (7,-7);
\draw [thick, black, dashed] (0,-8) -- (7,-8);
\foreach \i in {1,2,...,8,9} {
	\node at (-0.3,-\i+1) {$\i$};
	\node at (7.3,-\i+1) {$\i$};
}
\node at (3.5,1) {$t=5$};
\end{scope}
\end{tikzpicture}
\end{center}
\caption{The left-to-right line diagrams for the permutations rearranging $\alpha_{(5,2,2,1)}(t)+\varrho'$ into descending order when $t=0$ and $t=5$, where all crossings occur as early as possible. The diagram for $\lambda$ (vertically flipped then rotated by $45^\circ$ clockwise) is drawn on each side, with the $t$-interior shaded.}
\label{fig:permutation}
\end{figure}
\end{example}
\begin{remark}
For $t=0$, the map $\lambda\mapsto \sigma$ recovers the well-known bijection between $(k,n-k)$-bounded partitions and Grassmannian permutations in $\mathfrak{S}_n$ whose only descent (if it has one) is at $k$ (see \S19 of \cite{Pos06}).
\end{remark}

\begin{theorem}
For a bounded $t$-core partition $\lambda$, set $\iota:=\iota_\lambda(t)$. The BWB weight $\beta_\lambda(t)$ is given by rearranging the sequence
\[(-\lambda+\iota)\smile(\lambda^\T-\iota^\T-t)\]
into weakly decreasing order.
\end{theorem}
\begin{proof}
Set $\alpha:=\alpha_\lambda(t)$, $\beta:=\beta_\lambda(t)$ and recall that $\beta=\sigma(\alpha+\varrho')-\varrho'$. If $\sigma(i)=j$, then
\[\beta_j=\alpha_i+\varrho'_i-\varrho'_j=\alpha_i+(j-i).\]
By the calculation of $\sigma$ in \Cref{lem:sigma}, we have
\[\beta_{k-a+1+\iota_a}=\alpha_{k-a+1}+\iota_a=-\lambda_a+\iota_a\]
and
\[\beta_{k+b-\iota_b^\T}=\alpha_{k+b}-\iota_b^\T=\lambda_b^\T-t-\iota_b^\T\]
for every $1\leq a\leq k$ and $1\leq b\leq n-k$. Since $\sigma$ was a bijection, this gives every entry of $\beta$. These entries appear in weakly decreasing order by the definition of $\sigma$.
\end{proof}

\begin{corollary}
For a bounded $t$-core partition $\lambda$, we have $-\mathrm{rev}({\beta_\lambda(t)})=\beta_{\lambda^\C}(-t)$. If $t\geq 0$, then $-\mathrm{rev}(\beta_\lambda(t))\in\Lambda_n^+$.
\end{corollary}
\begin{proof}
We have $\iota_{\lambda^\C}(-t)=\iota_\lambda(t)^\C$. Therefore, the entries of $\beta_{\lambda^\C}(-t)$ are the union of the entries in
\[-\lambda^\C+\iota^\C=-(n-k-\mathrm{rev}(\lambda))+(n-k-\mathrm{rev}(\iota))=\mathrm{rev}(\lambda-\iota)\]
and
\[\lambda^{\C\T}-\iota^{\C\T}+t=(k-\mathrm{rev}(\lambda^\T))-(k-\mathrm{rev}(\iota^\T))+t=t-\mathrm{rev}(\lambda^\T-\iota^\T)\]
which are the entries of $-\mathrm{rev}(\beta_\lambda(t))$. Since both sequences are weakly decreasing, they must be equal. If $t>0$, the sequences $\lambda-\iota$ and $\lambda^\T-\iota^\T$ have entries in $\{0,1,\dots,t-1\}$ by \Cref{lem:bounded-bijection}, so $-\mathrm{rev}(\beta_\lambda(t))\in\Lambda_n^+$ since it has only nonnegative entries. If $t=0$, then $\iota=\lambda$ so $-\mathrm{rev}(\beta_\lambda(0))=(0^n)\in\Lambda_n^+$.
\end{proof}

As a consequence, $L_{\beta_\lambda(t)}^\vee$ is a polynomial representation for $t\geq 0$, and its dual is polynomial when $t\leq 0$ (when $t=0$ it is the trivial representation). In particular, the dimension $d_\lambda(t)$ in each case is the enumeration of semistandard tableaux of a given shape.

\begin{definition}\label{def:gamma}
For a bounded $t$-core partition $\lambda$, define $\gamma_\lambda(t)\in\Lambda_n^+$ by
\[\gamma_\lambda(t)=\begin{cases}-\mathrm{rev}(\beta_\lambda(t)) & \text{if }t\geq 0,\\\beta_\lambda(t) & \text{if }t< 0.\end{cases}\]
\end{definition}

By applying the above results, we can give an explicit enumeration for $d_\lambda(t)$ in terms of semistandard tableaux since for $t\geq 0$, $L_{\beta_\lambda(t)}^\vee\cong L_{-\mathrm{rev}(\beta_\lambda(t))}=L_{\gamma_\lambda(t)}$ is a polynomial representation. Recall the definition of the partition $\partial_\lambda(t):=\lambda-\iota_\lambda(t)$ for $t>0$ from \Cref{def:delta}. This recovers \Cref{thm:gamma-dimension-intro}.

\begin{theorem}\label{thm:gamma}
For a bounded $t$-core partition $\lambda$, $d_\lambda(t)=\#\SSYT(\gamma_\lambda(t);n)$. If $t>0$, then $\gamma_\lambda(t)$ is obtained by rearranging the sequence
\[\partial_\lambda(t)\smile (t-\partial_{\lambda^\T}(t))\]
into weakly descending order, and $\gamma_\lambda(-t)=\gamma_{\lambda^\C}(t)$. For all integers $t$, $\gamma_\lambda(t)$ is an $(n,|t|)$-bounded partition of size $(n-k)|t|$.
\end{theorem}

\begin{example}
Let $\lambda=(8,5,2,1)$ be the Snow partition with parameters $k=4$, $n=14$, $j=16$, $t=4$ and $i=7$.
\[\begin{tikzpicture}[scale = 0.6]
\draw [black!50!white] (0,0) grid (10,-4);
\draw [fill=intrColor] (0,0) -- (5,0) -- (5,-1) -- (2,-1) -- (2,-2) -- (0,-2) -- cycle;
\draw [thick] (0,0) grid (8,-1);
\draw [thick] (0,-1) grid (5,-2);
\draw [thick] (0,-2) grid (2,-3);
\draw [thick] (0,-3) grid (1,-4);
\end{tikzpicture}\]
We have $\partial_\lambda(4)=(3^2,2,1)$, $\partial_{\lambda^\T}(4)=(2,1^7,0^2)$ and $4-\partial_{\lambda^\T}(4)=(2,3^7,4^2)$, hence $\gamma_\lambda(4)=(4^2,3^9,2^2,1)$. Then $d_\lambda(4)=\#\SSYT(\gamma_\lambda(4);14)=354900$.
\end{example}

\subsection{A hook-product formula}\label{sect:dimension-formula}

In this section we prove \Cref{thm:dimension-formula-intro}, naturally extending the result to all integers $t$. Recall the definitions of $\alpha_\lambda(t)$, $\beta_\lambda(t)$ and $d_\lambda(t)$ from \Cref{sec:BWB}.

In this section, a Snow partition $\lambda$ will always have fixed parameters $(k,n,j,t,i)$.

\begin{lemma}
    \label{lem: snow partition dimension}
For a Snow partition $\lambda$,
\[d_\lambda(t)=(-1)^i\prod_{1\leq r<s\leq n}\frac{\alpha_\lambda(t)_r-\alpha_\lambda(t)_s+s-r}{s-r}.\] 
\end{lemma}
\begin{proof}
Set $\alpha:=\alpha_\lambda(t)$ and $\beta:=\beta_\lambda(t)$. By the Weyl dimension formula,
\[d_\lambda(t)=\dim L_\beta^\vee=\dim L_\beta=\prod_{\phi\in\Phi^+}\frac{\langle \phi,\beta+\varrho'\rangle}{\langle \phi,\varrho'\rangle}.\]
Since the inner product is invariant under reordering of the elements,
\[\langle \phi,\beta+\varrho'\rangle=\langle \phi,\sigma(\alpha+\varrho')\rangle=\langle \sigma^{-1}(\phi),\alpha+\varrho'\rangle.\]
The permutation $\sigma^{-1}$ acts bijectively on the set $\{\{\phi,-\phi\}\mid \phi\in\Phi^+\}$ and swaps a positive root for a negative root for each inversion of $\sigma$. Therefore, since $\sigma$ (and hence $\sigma^{-1})$ has $i$ inversions by \Cref{lem:unique-entries} we obtain
\[\dim L_\beta=(-1)^i\prod_{\phi\in\Phi^+}\frac{\langle \phi,\alpha+\varrho'\rangle}{\langle \phi,\varrho'\rangle}=(-1)^i\prod_{1\leq r<s\leq n}\frac{\alpha_r-\alpha_s+s-r}{s-r}.\qedhere\]
\end{proof}

We evaluate the product in the above lemma by splitting it into three parts. Two of these parts will then be consolidated using the result below.

\begin{lemma}\label{lem:skew-howe-dimension}
For a $(k,n-k)$-bounded partition $\lambda$,
\[\dim L_\lambda\cdot \dim L_{\lambda^\T}=(-1)^{j-N}\prod_{a=1}^k\prod_{b=1}^{n-k}\frac{a+b-1}{h_{\lambda}(a,b)}.\]
\end{lemma}
\begin{proof}
By \cite[\S2.1]{Gre18},
\[\dim L_\lambda\cdot \dim L_{\lambda^\T}=\frac{\prod_{(a,b)\in (n-k)^k} h_{k^{(n-k)}}(a,b)}{\prod_{(a,b)\in\lambda}h_\lambda(a,b)\prod_{(a,b)\in\lambda^\C}h_{\lambda^\C}(a,b)}.\]
The hook lengths in $\lambda^\C$ are precisely those that appear as negative hook lengths in $\lambda$ when ranging over all $1\leq a\leq k$ and $1\leq b\leq n-k$, of which there are $N-j$. This gives the denominator and the sign of our formula. For the numerator, we have
\[\prod_{a=1}^k\prod_{b=1}^{n-k}\bigl((n-k)+k-a-b+1\bigr) = \prod_{a=1}^k\prod_{b=1}^{n-k}(a+b-1)\]
after reindexing $a$ for $k-a+1$ and $b$ for $n-k-b+1$.
\end{proof}

\begin{lemma}
    \label{lem: weyl product to hook product}
For a Snow partition $\lambda$,
\[\prod_{1\leq r<s\leq n}\frac{\alpha_\lambda(t)_r-\alpha_\lambda(t)_s+s-r}{s-r}=(-1)^j\prod_{a=1}^k\prod_{b=1}^{n-k}\frac{h_{\lambda}(a,b)-t}{h_{\lambda}(a,b)}.\]
\end{lemma}
\begin{proof}
We break the product on the right into three parts. If both $r$ and $s$ are at most $k$, we obtain
\[\prod_{1\leq r<s\leq k}\frac{(-\lambda_{k-r+1})-(-\lambda_{k-s+1})+s-r}{s-r}\]
which after reindexing $r$ for $k-s+1$ and $s$ for $k-r+1$ becomes
\[\prod_{1\leq r<s\leq k}\frac{\lambda_r-\lambda_s+s-r}{s-r}=\dim L_{\lambda}.\]
Similarly, if both $r$ and $s$ are greater than $k$ we get
\[\prod_{k<r<s\leq n}\frac{(\lambda_{r-k}^\T-t)-(\lambda_{s-k}^\T-t)+s-r}{s-r}\]
and reindexing $r$ for $r-k$ and $s$ for $s-k$ gives
\[\prod_{1\leq r<s\leq n-k}\frac{\lambda_r^\T-\lambda_s^\T+s-r}{s-r}=\dim L_{\lambda^\T}.\]
Finally, we consider the mixed terms where $r\leq k$ and $s>k$. Reindexing $a$ for $k-r+1$ and $b$ for $s-k$ gives $s-r=a+b-1$ and hence
\[\prod_{a=1}^k\prod_{b=1}^{n-k}\frac{-\lambda_a-(\lambda_b^\T-t)+a+b-1}{a+b-1}=(-1)^N\prod_{a=1}^k\prod_{b=1}^{n-k}\frac{h_\lambda(a,b)-t}{a+b-1}.\]
Multiplying the three terms together and directly applying \Cref{lem:skew-howe-dimension} gives the required result.
\end{proof}
We can now state and immediately prove the main theorem for this section, which extends \Cref{thm:dimension-formula-intro} to all integers $t$.

\begin{definition}\label{def:t-hook-ratio}
Suppose $\lambda$ is a $(k,n-k)$-bounded partition and $t$ is an integer. We define the $t$-\emph{hook ratio} of $\lambda$ to be
\[f_\lambda(t):=\prod_{a=1}^k\prod_{b=1}^{n-k}\frac{h_\lambda(a,b)-t}{h_\lambda(a,b)},\]
where $h_\lambda(a,b):=\lambda_a+\lambda_b^\T-a-b+1$.
\end{definition}
\begin{theorem}\label{thm:dimension-formula}
For a Snow partition $\lambda$, $d_\lambda(t)=(-1)^{i+j}f_\lambda(t)$. Hence,
\[h^{j,i}(t)=(-1)^{i+j}\sum_\lambda f_\lambda(t)\]
where the sum is over all Snow partitions with parameters $(k,n,j,t,i)$.
\end{theorem}
\begin{proof}
    This follows from \Cref{thm:cohomology-iso} and \Cref{lem: snow partition dimension,lem: weyl product to hook product}.
\end{proof}
We conclude this section with some comments on the $t$-hook ratio. For a bounded partition $\lambda$, $f_\lambda(t)$ is a rational polynomial whose constant term is $1$, and whose roots are the (positive and negative) hook lengths of $\lambda$, counted with multiplicity. When $\lambda=\lambda^\C$, then $f_\lambda(t)$ can be written
\[\prod_{(a,b)\in\lambda}\left(1-\frac{t^2}{h_\lambda(a,b)^2}\right)\]
which appears in \cite[Theorem 1.1]{Han08} and is closely related to the Nekrasov-Okounkov hook length formula \cite[Theorem 6.12]{Nek06}.

\begin{remark}\label{rmk:integer-valued}
Extend the definition of $d_\lambda(t)$ to be zero when $\lambda$ is not $t$-core, in which case $f_\lambda(t)$ is also zero. Then $d_\lambda(t)=|f_\lambda(t)|$ for every bounded partition $\lambda$ and integer $t$. In particular, the $t$-hook ratio is an integer-valued polynomial in $t$. However, $d_\lambda(t)$ and $h^{j,i}(t)$ are not polynomials in general, since the sign $(-1)^{i+j}$ (and the $i$ term in particular) depends on both $\lambda$ and $t$.
\end{remark}

\subsection{Examples}\label{sec:examples}

To compute the Hodge numbers $h^{j,i}(t)$ by the above methods, we first have to classify all Snow partitions $\lambda$ with parameters $(k,n,j,t,i)$. We then have three options for computing $d_\lambda(t)$: using the Weyl dimension formula on $\beta_\lambda(t)$ directly requires a product of $n(n-1)/2$ ratio terms, calculating $\#\SSYT(\gamma_\lambda(t);n)$ by \Cref{eq:SSYT} requires $(n-k)|t|$ terms, and evaluating $f_\lambda(t)$ requires $N=k(n-k)$ terms. The latter calculation can sometimes be dramatically simplified via a simple observation.

\begin{lemma}\label{lem:h+h'}
Suppose $h$ and $h'$ satisfy $h+h'=t$. Then
\[\frac{h-t}{h}=\left(\frac{h'-t}{h'}\right)^{-1},\]
so the ratios corresponding to boxes in $\lambda$ with hook lengths $h$ and $h'$ cancel in the evaluation of $f_\lambda(t)$.
\end{lemma}

\begin{example}[Duality]
Since $\lambda$ and $\lambda^\T$ have the same hook lengths, the $t$-hook ratios $f_\lambda(t)$ and $f_{\lambda^\T}(t)$ are equal. This arises from the duality between $\Gr(k,n)$ and $\Gr(n-k,n)$. The hook lengths of $\lambda^\C$ are the negatives of the hook lengths of $\lambda$, which gives the equality $f_\lambda(-t)=f_{\lambda^\C}(t)$. This gives $h^{j,i}(t)=h^{N-j,N-i}(-t)$, which is implied by Serre duality.

From these equalities, when calculating $h^{j,i}(t)$ we can assume that $t\geq 0$ and $k\leq n/2$ without loss of generality.
\end{example}

\begin{example}[$j=0$]
The only Snow partition is $\lambda=(0^k)$. For $t\geq 0$, we have $i=0$ and $h_\lambda(a,b)=-a-b+1$ for every box $(a,b)$, so
\[h^{0,0}(t)=d_\lambda(t)=\prod_{a=1}^k\prod_{b=1}^{n-k}\frac{-a-b+1-t}{-a-b+1}=\prod_{a=1}^k\prod_{b=1}^{n-k}\frac{a+b+t-1}{a+b-1}\]
which is the number of $(k,n-k,t)$-bounded plane partitions \cite{Mac86}.

Alternatively, for $t\geq 0$ we have $\gamma_\lambda(t)=(t^{n-k},0^k)$, so $f_\lambda(t)=\#\SSYT((t^{n-k});n)$. Subtracting $\kappa$ from every label in row $\kappa$ of the tableau for every $1\leq \kappa\leq n-k$ gives a bijection to plane partitions which fit inside a $(n-k)\times t\times k$ box, since entries now weakly decrease along rows and down columns.

\label{ex: Macmahon}
\end{example}

\begin{example}[$t=1$]
The only $1$-core partition is again $\lambda=(0^k)$, so
\begin{align*}h^{0,0}(1)&=\#\{\text{plane partitions in a $k\times(n-k)\times 1$ box}\}\\
&=\#\{\text{bounded partitions}\}=\binom{n}{k},\end{align*}
and $h^{j,i}(1)=0$ otherwise.
\end{example}

\begin{example}[$t=0$]
    \label{ex:t=0}
Every bounded partition $\lambda$ is $0$-core with $i=j$. The BWB weight is $\beta_\lambda(0)=(0^n)$ and hence $d_\lambda(0)=1$. Equivalently, $f_\lambda(0)=1$ since every term in the product is $1$. Therefore,
\[h^{j,j}(0)=\#\{\text{bounded partitions of size $j$}\},\]
which is the coefficient of $q^j$ in the Gaussian binomial coefficient
\[\frac{(1-q^n)\cdots(1-q)}{(1-q^k)\cdots (1-q)(1-q^{n-k})\cdots(1-q)}\in\Z[q].\]
Otherwise, $h^{j,i}(0)=0$ for $i\neq j$.
\end{example}

\begin{example}[$t=2$]
    \label{ex:t=2}
Assume without loss of generality that $k\leq n/2$. The bounded $2$-core partitions are the staircase partitions $\Delta_m:=(m,\dots,2,1)$ for every $0\leq m\leq k$. For a fixed such $m$ set $\beta:=\beta_{\Delta_m}(2)$. Since $\iota_{\Delta_m}(2)=\Delta_{m-1}$, we have
\[\gamma_{\Delta_m}(2)=(2^{n-k-m},1^{2m},0^{k-m}).\]
Applying \Cref{eq:SSYT} to this shape gives
\[d_{\Delta_m}(2)=\frac{n\cdots(n-k-m+1)(n+1)\cdots(n-k+m+2)}{(2m)!(k-m)!(k+m+1)\cdots(2m+2)}\]
which can be simplified to
\[\frac{2m+1}{n+1}\binom{n+1}{k+m+1}\binom{n+1}{n-k+m+1},\]
recovering \cite[Theorem 3.3]{Sno86}. In total, we obtain
\[h^{m(m+1)/2,m(m-1)/2}(2)=\frac{2m+1}{n+1}\binom{n+1}{k+m+1}\binom{n+1}{n-k+m+1}\]
for every $0\leq m\leq \min(k,n-k)$, and $h^{j,i}(2)=0$ otherwise.
\end{example}

\section{Restrictions on bounded core partitions}

We have seen that positivity of the twisted Hodge number $h^{j,i}(t)$ is equivalent to the existence of a Snow partition $\lambda$. In this section we prove necessary conditions for the existence of Snow partitions with given parameters. When these restrictions are not satisfied, the corresponding cohomology will vanish.

\subsection{Skew-linking diagrams}\label{sec:plane-partitions}

We use a construction of Chen \cite{Chen10} to define a map $\lambda\mapsto\mathcal{P}_\lambda(t)$ from Snow partitions to plane partitions (which were defined in \Cref{sec:tableaux}). Many of our restrictions on Snow partitions then derive from simple arguments about plane partitions.

Recall that for a Snow partition $\lambda$ with $t>0$, the skew partition $\lambda/\iota_\lambda(t)$ has weakly decreasing row and column lengths by \Cref{lem:bounded-bijection}. Skew partitions with this property are called \emph{skew-linking diagrams}. Note that there are skew-linking diagrams which are not of the form $\lambda / \iota_\lambda(t)$.

We build a plane partition $P$ out of an arbitrary skew-linking diagram $\lambda/\mu$ with an iterative procedure. In the first step, read down the rows of $\lambda/\mu$ and select any row which does not overlap in the columns of an already-selected row. We highlight these rows in the below skew-linking diagram, which were selected top-to-bottom:

\[\begin{tikzpicture}[scale=0.4]
\fill [shade1] (4,0) rectangle (9,-1);
\fill [shade1] (1,-2) rectangle (4,-3);
\fill [shade1] (0,-5) rectangle (1,-6);

\draw [thick] (4,0) grid (9,-1);
\draw [thick] (1,-2) grid (4,-3);
\draw [thick] (0,-5) grid (1,-6);
\draw [thick] (2,-1) grid (6,-2);
\draw [thick] (0,-4) grid (2,-5);
\draw [thick] (1,-3) grid (3,-4);
\draw [thick] (0,-6) grid (1,-7);
\end{tikzpicture}\]

The labels of the first row of $P$ are determined by `stacking' these rows. More precisely, $P_{1,s}$ for $s\geq 1$ is the number of rows selected in this step with length at least $s$. We then label the second row of $P$ by the same procedure, but ignore any rows in $\lambda/\mu$ which have already been selected. Below we shade the selected rows in the second step.

\[\begin{tikzpicture}[scale=0.4]
\fill [shade1] (4,0) rectangle (9,-1);
\fill [shade1] (1,-2) rectangle (4,-3);
\fill [shade1] (0,-5) rectangle (1,-6);
\fill [shade2] (2,-1) rectangle (6,-2);
\fill [shade2] (0,-4) rectangle (2,-5);

\draw [thick] (4,0) grid (9,-1);
\draw [thick] (1,-2) grid (4,-3);
\draw [thick] (0,-5) grid (1,-6);
\draw [thick] (2,-1) grid (6,-2);
\draw [thick] (0,-4) grid (2,-5);
\draw [thick] (1,-3) grid (3,-4);
\draw [thick] (0,-6) grid (1,-7);
\end{tikzpicture}\]

We continue until all rows have been exhausted, at which point we obtain $P$.

\[\begin{tikzpicture}[scale = 0.4]
\fill [shade1] (4,0) rectangle (9,-1);
\fill [shade1] (1,-2) rectangle (4,-3);
\fill [shade1] (0,-5) rectangle (1,-6);
\fill [shade2] (2,-1) rectangle (6,-2);
\fill [shade2] (0,-4) rectangle (2,-5);
\fill [shade3] (1,-3) rectangle (3,-4);
\fill [shade3] (0,-6) rectangle (1,-7);

\draw [thick] (4,0) grid (9,-1);
\draw [thick] (1,-2) grid (4,-3);
\draw [thick] (0,-5) grid (1,-6);
\draw [thick] (2,-1) grid (6,-2);
\draw [thick] (0,-4) grid (2,-5);
\draw [thick] (1,-3) grid (3,-4);
\draw [thick] (0,-6) grid (1,-7);

\fill [shade1] (12,0) rectangle (23,-2);
\fill [shade2] (12,-2) rectangle (23,-4);
\fill [shade3] (12,-4) rectangle (23,-6);
\node at (13,-1) {\Large 3};
\node at (15,-1) {\Large 2};
\node at (17,-1) {\Large 2};
\node at (19,-1) {\Large 1};
\node at (21,-1) {\Large 1};
\node at (13,-3) {\Large 2};
\node at (15,-3) {\Large 2};
\node at (17,-3) {\Large 1};
\node at (19,-3) {\Large 1};
\node at (13,-5) {\Large 2};
\node at (15,-5) {\Large 1};

\foreach \i in {0,-2,-4,-6} {
	\draw (12,\i) -- (23,\i);
}
\foreach \i in {12,14,16,18,20,22} {
	\draw (\i,0) -- (\i,-7);
}
\end{tikzpicture}\]

We could create another plane partition $P'$ by swapping the roles of rows and columns in the construction, and labelling $P'$ one column at a time by selecting columns in $\lambda/\mu$ which do not overlap in their rows. Surprisingly, it turns out that $P=P'$ \cite[Corollary 3.4.1]{Chen10}. Moreover, we obtain the following properties:

\begin{theorem}
Let $\lambda/\mu$ be a skew-linking diagram, and let $P$ be the plane partition obtained from it by the above construction. Then:
\begin{itemize}
\item The length of the first row/column of $P$ is the length of the first row/column of $\lambda/\mu$.
\item The sum of the first row/column of $P$ is the number of columns/rows in $\lambda/\mu$.
\item The sum $|P|$ of labels in $P$ is equal to the number of boxes in $\lambda/\mu$.
\item Let $\Delta P$ be the plane partition obtained by replacing every label $\ell$ in $P$ by $1+2+\cdots+(\ell-1)$. Then $|\Delta P|=|\mu|$.
\end{itemize}
\end{theorem}
\begin{proof}
The first (and longest) row of $\lambda/\mu$ is always the first row selected in the construction, which adds this many $1$s to the first row of $P$. The statement for columns follows from $P=P'$. The first box of every row in $\lambda/\mu$ increments a label in the first column of $P$. Similarly for columns and rows by $P=P'$.

Every box of $\lambda/\mu$ increments a label in $P$, which gives the third statement. For the fourth, see \cite[Lemma 2.3.6]{Chen10}.
\end{proof}

\begin{definition}
For a Snow partition $\lambda$ with $t>0$, let $\mathcal{P}_\lambda(t)$ denote the plane partition obtained from $\lambda/\iota_\lambda(t)$ using the above construction.
\end{definition}

\begin{corollary}\label{cor:plane-partition}
For a Snow partition $\lambda$ with parameters $(k,n,j,t,i)$ and $t>0$, let $P:=\mathcal{P}_\lambda(t)$. Then:
\begin{itemize}
\item $P$ is $(t-1,t-1,k)$-bounded.
\item The sum of the first column of $P$ is at most $k$.
\item The sum of the first row in $P$ is at most $n-k$.
\item If $\Delta P$ is defined as above, then $i=|\Delta P|$ and $j=|P|+|\Delta P|$.
\end{itemize}
\end{corollary}
\begin{proof}
Recall that $\lambda/\iota(t)$ is a skew-linking diagram whose first row and first column have length at most $t-1$ by \Cref{lem:bounded-bijection}. Also, the number of rows and columns in $\lambda$ are at most $k$ and $n-k$. Finally, we have $j=i+(j-i)=|\Delta P|+|P|$.
\end{proof}

\subsection{Nakano vanishing}
\label{sec: nakano}

The Nakano vanishing theorem (see \cite{Nak54}), sometimes called the Kodaira--Nakano or Akizuki--Nakano vanishing theorem, asserts that for a compact K\"ahler manifold $X$ and an ample line bundle $L$, 
\[H^i(X,\Omega_X^j\otimes L)=0\quad \text{ whenever }\quad i+j>\dim X.\]
In the case $X=\mathrm{Gr}(k,n)$, this gives $H^i(X,\Omega^j(t))=0$ whenever $t>0$ and $i+j>N$, where $N:=k(n-k)$. We prove this inequality for all plane partitions, where the values $k,n,j,i$ are associated to a plane partition based on \Cref{cor:plane-partition}. We then obtain the same inequality for Snow partitions using the map $\lambda\mapsto\mathcal{P}_\lambda(t)$ from the previous section.

\begin{lemma}\label{lem:nakano}
For a plane partition $P$, choose $k,n,j,i$ such that
\begin{itemize}
\item $k$ is the sum of the first column of $P$,
\item $n-k$ is the sum of the first row of $P$,
\item $i=|\Delta P|$ and $j=|\Delta P|+|P|$.
\end{itemize}
Then $i+j\leq k(n-k)$.
\end{lemma}
\begin{proof}
We have
\[k(n-k)=\sum_{r}P_{r,1}\sum_s P_{1,s}\]
by definition. Since the entries in a plane partition are weakly decreasing along rows and down columns, we get
\[\sum_r P_{r,1}\sum_s P_{1,s}\geq \sum_{r,s}P_{r,s}^2\]
and by the identity $\ell^2=2(1+2+\cdots+(\ell-1))+\ell$, the right-hand side of this inequality is equal to $2|\Delta P|+|P|=i+j$.
\end{proof}
\begin{remark}\label{rmk:tightly-bounded}
In general, the values of $k$ and $n-k$ for the Snow partition $\lambda$ might be greater than the associated values for $\mathcal{P}_\lambda(t)$, with equality when $\lambda$ is \emph{tightly} bounded in its $k\times(n-k)$ rectangle, or equivalently, when $\lambda$ and $\lambda^\T$ have no zero entries. The proof using plane partitions is sufficient because increasing $k$ and $n-k$ preserves the inequality $i+j\leq k(n-k)$.
\end{remark}

In the next theorem, we calculate the Hodge numbers when $i+j=N$ and observe that the Snow partitions which satisfy this condition are ``generalized staircases'' (see \Cref{fig:i+j=N}).

\begin{figure}[H]
\begin{center}
\begin{tikzpicture}[scale = 0.4]
\node at (4,-5) {$t=3$};
\draw [black!50!white] (0,0) grid (8,-4);
\fill [intrColor] (0,0) -- (6,0) -- (6,-1) -- (4,-1) -- (4,-2) -- (2,-2) -- (2,-3) --  (0,-3) -- cycle;
\draw [thick] (0,0) grid (2,-4);
\draw [thick] (2,0) grid (4,-3);
\draw [thick] (4,0) grid (6,-2);
\draw [thick] (6,0) grid (8,-1);
\end{tikzpicture}
\quad
\begin{tikzpicture}[scale = 0.4]
\node at (4,-5) {$t=6$};
\draw [black!50!white] (0,0) grid (8,-4);
\fill [intrColor] (0,0) -- (4,0) -- (4,-2) -- (0,-2) -- cycle;
\draw [thick] (0,0) grid (4,-4);
\draw [thick] (4,0) grid (8,-2);
\end{tikzpicture}\quad
\begin{tikzpicture}[scale = 0.4]
\node at (4,-5) {$t\geq 12$};
\fill [intrColor] (0,0) -- (8,0) -- (8,-4) -- (0,-4) -- cycle;
\draw [thick] (0,0) grid (8,-4);
\end{tikzpicture}
\end{center}
\caption{All Snow partitions for $t>0$ with $k=8$, $n=12$ and $i+j=N$. The $t$-interior of each is shaded with the possible values of $t$ given.}
\label{fig:i+j=N}
\end{figure}

\begin{theorem}\label{thm:i+j=N}
Suppose $i,t>0$ and $i+j=N$. Then $h^{j,i}(t)\neq 0$ if and only if $t\mid n\mid kt$ and $i=N(n-t)/(2n)$. In this case, $h^{j,i}(t)=1$.
\end{theorem}
\begin{proof}
To get equality in the proof of \Cref{lem:nakano}, we need $\lambda$ to be tightly bounded (\Cref{rmk:tightly-bounded}), and for $P_{r,s}^2=P_{r,1}P_{1,s}$ for all $r$ and $s$, which means $P$ takes a constant value $\ell$ on some $a\times b$ rectangle with no other labels. The skew-linking diagram which maps to this partition consists of $\ell$ diagonally arranged blocks of size $a\times b$. We have $k=\ell a$ and $n-k=\ell b$, so $n=\ell(a+b)$.

The largest hook length on each diagonal block is $a+b-1$, and the smallest hook lengths above the diagonal blocks are $a+b+1$, so $t=a+b$ and $n=\ell t$. Together we get $t\mid n\mid kt$. The size of the interior is
\[i=\frac{ab\ell(\ell-1)}{2}=\frac{N(\ell-1)}{2\ell}=\frac{N(n-t)}{2n}.\]
Conversely, if $t\mid n\mid kt$ we set $\ell=n/t$, $a=k/\ell$, $b=(n-k)/\ell$, which are integers. We let $\lambda$ be the generalized staircase with $\ell$ diagonal blocks of size $a\times b$.

We can prove that $d_\lambda(t)=1$ either using the $t$-hook ratio and the cancelling of boxes in $\iota_\lambda(t)$ and $\lambda^\C$ using \Cref{lem:h+h'}, or by observing that $\gamma_\lambda(t)=(b^n)$ and $\#\SSYT((b^n);n)=1$ since every column must contain all values $1,\dots,n$.
\end{proof}

\begin{example}
For $k=2$ and $n=4$ (so $N=4$), the only Snow partition with $i+j=4$ is $\lambda=(2,1)$ with $j=3$ and $i=1$. Below is the Hodge diamond for $\mathrm{Gr}(2,4)$ twisted by $\mathcal{O}(2)$, showcasing $h^{3,1}(2)=1$.
\[\begin{tikzpicture}[scale=.4, every node/.style={inner sep=1pt}]

    \node at (0,4) {0};

    \node at (-1,3) {0};
    \node at (1,3) {0};

    \node at (-2,2) {0};
    \node at (0,2) {0};
    \node at (2,2) {0};

    \node at (-3,1) {0};
    \node at (-1,1) {0};
    \node at (1,1) {0};
    \node at (3,1) {0};

    \node at (-4,0) {0};
    \node at (-2,0) {1};
    \node at (0,0) {0};
    \node at (2,0) {0};
    \node at (4,0) {0};

    \node at (-3,-1) {0};
    \node at (-1,-1) {0};
    \node at (1,-1) {0};
    \node at (3,-1) {0};

    \node at (-2,-2) {0};
    \node at (0,-2) {0};
    \node at (2,-2) {0};

    \node at (-1,-3) {15};
    \node at (1,-3) {0};

    \node at (0,-4) {20};

    \draw[thick] (-4.6,-0.6) rectangle (4.6,0.6);
    \node[font=\small] at (4.8,0) [anchor=west] {$\,\,i+j=4$};
\end{tikzpicture}\]
\end{example}

The case $i+j=N-1$ turns out to be far more restrictive. In this case we get at most two Snow partitions satisfying $i,t>0$ and $i+j=N-1$ for any fixed value of $n$ (see \Cref{fig:i+j=N-1}).

\begin{figure}[H]
\begin{center}\begin{tikzpicture}[scale = 0.6]
\draw [black!50!white] (0,0) grid (9,-2);
\fill [intrColor] (0,0) -- (4,0) -- (4,-1) -- (0,-1) -- cycle;
\draw [thick] (0,0) grid (4,-2);
\draw [thick] (4,0) grid (9,-1);
\end{tikzpicture}\end{center}
\caption{A Snow partition $\lambda$ satisfying $i+j=N-1$ for $n=11$ and $i>0$. The only other such Snow partition is $\lambda^\T$. Both require that $i=4$ and $t=6$.}
\label{fig:i+j=N-1}
\end{figure}

\begin{theorem}\label{thm:i+j=N-1}
Suppose $i,t>0$ and $i+j=N-1$. Then $h^{j,i}(t)=n$ if $n\geq 5$, $k\in\{2,n-2\}$, $i=(n-3)/2$ and $t=(n+1)/2$. Otherwise, $h^{j,i}(t)=0$.
\end{theorem}
\begin{proof}
Fix $1\leq k<n$ and $t>0$. We first prove that there are at most two Snow partitions when $i+j=N-1$ and $i>0$.

Consider the chain of inequalities in \Cref{lem:nakano}. If $i>0$, we have both $k>1$ and $n-k>1$. Hence, if $\lambda$ or $\lambda^\T$ had zero entries, we would have $k>\sum_r P_{r,1}$ or $n-k>\sum_s P_{1,s}$, and $k(n-k)\geq i+j+2$ in either case.

Hence, we have $\sum_{r,s}P_{r,1}P_{1,s}=\sum_{r,s}P_{r,s}^2+1$, so
\[P=\vcenter{\hbox{\begin{tikzpicture}[scale=0.6]
\draw (0,0) grid (6,-1);
\node at (0.5,-0.5) {2};
\node at (1.5,-0.5) {2};
\node at (2.5,-0.5) {2};
\node at (3.5,-0.5) {$\,\cdots$};
\node at (4.5,-0.5) {2};
\node at (5.5,-0.5) {1};
\end{tikzpicture}}}\]
or its transpose. If there are $r$ copies of $2$ in this plane partition, then $\lambda=(2r+1,r)$ (or its transpose) which is a Snow partition for $n=2r+3$, $i=r$ and $t=r+2$, satisfying the conditions in the theorem. Conversely, given parameters matching the conditions in the theorem we can set $r=t-2$ and $\lambda=(2r+1,r)$ when $k=2$, or its transpose when $k=n-2$, which both satisfy $i+j=N-1$.

We can calculate that $d_\lambda(t)=n$ for $\lambda=(2r+1,r)$ either using the $t$-hook ratio, or by observing that $\gamma_\lambda(t)=((r+1)^{n-1},r)$, and $\#\SSYT(\gamma_\lambda(t);n)=n$ since each tableau is determined by the entry missing from the final column.
\end{proof}

\begin{example}
For $k=2$ and $n=5$ (so $N=6$), the only Snow partition satisfying $i,t>0$ and $i+j=6-1$ is $\lambda=(3,1)$ with $j=4$ and $i=1$. Below is the Hodge diamond for $\mathrm{Gr}(2,5)$ twisted by $\mathcal{O}(3)$, showcasing $h^{4,1}(3)=5$.
\[\begin{tikzpicture}[scale=.4, every node/.style={inner sep=1pt}]

        \node at (0,6) {0};

        \node at (-1,5) {0};
        \node at (1,5) {0};

        \node at (-2,4) {0};
        \node at (0,4) {0};
        \node at (2,4) {0};

        \node at (-3,3) {0};
        \node at (-1,3) {0};
        \node at (1,3) {0};
        \node at (3,3) {0};

        \node at (-4,2) {0};
        \node at (-2,2) {0};
        \node at (0,2) {0};
        \node at (2,2) {0};
        \node at (4,2) {0};

        \node at (-5,1) {0};
        \node at (-3,1) {0};
        \node at (-1,1) {0};
        \node at (1,1) {0};
        \node at (3,1) {0};
        \node at (5,1) {0};

        \node at (-6,0) {0};
        \node at (-4,0) {0};
        \node at (-2,0) {0};
        \node at (0,0) {0};
        \node at (2,0) {0};
        \node at (4,0) {0};
        \node at (6,0) {0};

        \node at (-5,-1) {0};
        \node at (-3,-1) {5};
        \node at (-1,-1) {0};
        \node at (1,-1) {0};
        \node at (3,-1) {0};
        \node at (5,-1) {0};

        \node at (-4,-2) {0};
        \node at (-2,-2) {0};
        \node at (0,-2) {0};
        \node at (2,-2) {0};
        \node at (4,-2) {0};

        \node at (-3,-3) {0};
        \node at (-1,-3) {0};
        \node at (1,-3) {0};
        \node at (3,-3) {0};

        \node at (-2,-4) {120};
        \node at (0,-4) {0};
        \node at (2,-4) {0};

        \node at (-1,-5) {280};
        \node at (1,-5) {0};

        \node at (0,-6) {175};

        \draw[thick] (-5.6,-1.6) rectangle (5.6,-0.4);
        \node[font=\small] at (5.8,-1) [anchor=west] {$\,\,i+j=5$};

        \end{tikzpicture}\]
\end{example}

\subsection{Snow's bounds}\label{sect:bounds}

In \cite[\S3.4]{Sno86}, Snow lists five inequalities in terms of the parameters $(k,n,j,t,i)$ which guarantee the nonexistence of Snow partitions. In this section we strictly improve four of these bounds by turning them into necessary inequalities for the existence of Snow partitions. The bounds we provide are sharp in the sense that there are infinitely many Snow partitions satisfying the associated equality. Our proofs are given on the level of plane partitions where possible, using the construction in \Cref{sec:plane-partitions}.

Recall that by applying the symmetries $\lambda\mapsto\lambda^\T$ and $\lambda\mapsto\lambda^\C$ we can effectively assume that $t>0$ and $k\leq n/2$. Under these assumptions, Snow proves that no Snow partitions exist if any of the following are satisfied:

\begin{enumerate}[label=(S\arabic*)]
\item $ki\geq(k-1)j>0$, \medskip\label{eq:s1}
\item $i>N-j$, \medskip \label{eq:s2}
\item $j>k(n-k-1)\text{ if }(k,n)\neq (2,4)$, \medskip \label{eq:s3}
\item $j\leq t\text{ and }i>0$, \medskip \label{eq:s4}
\item $(2k-1)i\geq (k-1)N$. \medskip \label{eq:s5}
\end{enumerate}

We now state improved versions of these bounds. These hold under the assumption that $t>0$. Note that the inequalities are reversed, since these are necessary conditions for the existence of Snow partitions.
\begin{enumerate}[label=(P\arabic*)]
\item $i\leq\dfrac{k-1}{k+1}j$, \medskip\label{eq:p1}
\item $i+j\leq N$, \medskip \label{eq:p2}
\item $j\leq\dfrac{1}{2}k(n-k+t-1)$, \medskip \label{eq:p3}
\item $j-i\geq t$ if $i>0$, \medskip \label{eq:p4}
\item $i\leq\dfrac{k-1}{2k}N$. \medskip \label{eq:p5}
\end{enumerate}

\begin{example}The table below provides examples of parameters $(k,n,j,t,i)$ for which no Snow partitions exist. Each of these parameter sets is ruled out by our constraints \ref{eq:p1}--\ref{eq:p5}, but not by Snow's \ref{eq:s1}--\ref{eq:s5}.

\begin{center}
\begin{tabular}{|c c c c c| c c c c c|}
    \hline
    $k$ & $n$ & $j$ & $t$ & $i$ & \ref{eq:p1} & \ref{eq:p2} & \ref{eq:p3} & \ref{eq:p4} & \ref{eq:p5}\\
    \hline
    2&7&5&2&2 &False&True&True&True&True\\
    2&4&3&2&1 & True & True & False & True & True\\
    3&6&4&3&2 & True & True & True & False & True\\
    2&7&7&2&3 & False & True & False & True & False\\
    \hline
\end{tabular}
\end{center}
\end{example}

\begin{proposition}All Snow partitions with $i,t>0$ satisfy \emph{\ref{eq:p4}}.\end{proposition}
\begin{proof}
Choose a box $(a,b)\in\iota_\lambda(t)$ whose south and east neighbors are not in the $t$-interior. Then all boxes in the hook with vertex $(a,b)$ except for the vertex itself lie outside the $t$-interior, so
\[j-i\geq h_\lambda(a,b)-1\geq t+1-1\geq t.\qedhere\]
\end{proof}

We prove the remaining inequalities on the level of plane partitions by the same approach as \Cref{lem:nakano}. These inequalities then hold for all Snow partitions by the same reasoning (also see \Cref{rmk:tightly-bounded}, which applies to these inequalities in the same way).

\begin{proposition}
Every $(t-1,t-1)$-bounded plane partition $P$ satisfies \emph{\ref{eq:p1}}, \emph{\ref{eq:p2}}, \emph{\ref{eq:p3}} and \emph{\ref{eq:p5}}, with $k,n,j,i$ as in \Cref{lem:nakano}. Moreover, we have
\[i\leq\frac{\sqrt{N}(\sqrt{N}-1)}{2}.\]
\end{proposition}
\begin{proof}
\Cref{lem:nakano} gives $i+j\leq N$, which is \ref{eq:p2}.

We have the inequality $j-i\leq k(t-1)$ by replacing each column of $P$ by the first column. Combining this with $i+j\leq N$ gives \ref{eq:p3}.

We have $j-i\geq\max(k,n-k)$ (the total sum of entries compared to the first row/column) so with $i+j\leq N$ we get $i+j\leq k(j-i)$, which after rearranging gives \ref{eq:p1}. Applying $i+j\leq N$ to replace $j$ with $N-i$ in \ref{eq:p1} and rearranging gives \ref{eq:p5}. 

Replacing both $k$ and $n-k$ with $j-i$ in $i+j\leq N$ gives $i+j\leq (j-i)^2$. By finding the intersection of this region with $i+j\leq N$ (see the diagram below) we obtain $i\leq\sqrt{N}(\sqrt{N}-1)/2$.
\[\begin{tikzpicture}[scale = 0.4]
\fill [black!10!white, samples=300, domain=1:3.618] (1,0) --
    plot (\x,{0.5*(1+2*\x-sqrt(1+8*\x))}) -- (5,0) -- cycle;

\draw [->] (-0.5,0) -- (6,0) node [below right] {$j$};
\draw [->] (0,-0.5) -- (0,6) node [above left] {$i$};

\draw [thick] (-0.2,5.2) -- (5.2,-0.2);
\draw [thick, ->, samples=300, domain=1:4.9] plot (\x,{0.5*(1+2*\x-sqrt(1+8*\x))});

\draw (-0.2,5) -- (0.2,5) node [left, xshift = -2pt] {\small $N$};
\draw (5,-0.2) -- (5,0.2) node [below, yshift = -3pt] {\small $N$};
\end{tikzpicture}\qedhere\]
\end{proof}

All bounds in the above proposition are sharp for at least the staircase partitions $(k,\dots,2,1)$ where $n=2k$ and $t=2$.

\section{Hook Lengths and Euler Characteristic}\label{sec:euler}
\subsection{Algebraic proof}
We give an algebraic proof that the sum of $t$-hook ratios over all bounded partitions is independent of $t$. Recall $f_\lambda(t)$ from \Cref{def:t-hook-ratio}.

\begin{theorem}
\label{thm: euler characteristic hook length}
For every $t\in\mathbb{Z}$,
\[\sum_\lambda f_\lambda(t)=\binom{n}{k}\]
with the sum over all $(k,n-k)$-bounded partitions.
\end{theorem}
\begin{proof}
We define the twisted holomorphic Euler characteristic (see \cite[\S2.10]{Hir78}) of $\text{Gr}(k,n)$:
\[\chi(\text{Gr}(k,n),\Omega^j(t)) := \sum_{i\geq 0}(-1)^ih^i(\text{Gr}(k,n),\Omega^j(t)).\]
When $t=0$, the holomorphic Euler characteristic is related to the standard (topological) Euler characteristic $\chi(\text{Gr}(k,n))$ via 
\[\chi(\text{Gr}(k,n)) = \sum_{j\geq 0}(-1)^j \chi(\text{Gr}(k,n),\Omega^j).\]
One readily shows that $\chi(\text{Gr}(k,n)) = \binom{n}{k}$.
Now, let $X=\text{Gr}(k,n)$ and consider a generic section $\xi\in H^0(T_{X})$. This vanishes at $\binom{n}{k}$ points, corresponding to the zero locus $D$ (see e.g. \cite[\S14.1]{Ful84}). We consider the Koszul resolution (see e.g. \cite[\S17.2]{Eis95}) of the sheaf $\mathcal O_D$:
\[0\to\Omega_X^{k(n-k)} \xrightarrow{\iprod \xi}\cdots \xrightarrow{\iprod \xi} \Omega_X^2\xrightarrow{\iprod \xi}\Omega_X^1 \xrightarrow{\iprod \xi}\mathcal O_X \xrightarrow{} \mathcal O_{D}\to 0\]
which is exact because $\xi$ is a regular section. By the additivity of the Euler characteristic,
\[\sum(-1)^{j}\chi(\Omega_X^j) = \chi (\mathcal O_{D}) = \binom{n}{k}.\]
We note that $\chi (\mathcal O_{D}) = \chi(\mathcal O_X(t)\otimes \mathcal O_D)$ since $\mathcal O_D$ is a sum of skyscraper sheaves and $\mathcal O_X(t)$ is a line bundle, hence tensoring the Koszul resolution by $\mathcal O_X(t)$ we obtain
\[\sum_{j\geq 0}(-1)^{j}\chi(\Omega_X^j(t))=\binom{n}{k}.\]
By \Cref{thm:dimension-formula},
\[\sum_{|\lambda|=j} f_\lambda(t)=(-1)^j\chi(\Omega^j(t))\]
with the sum over all bounded partitions of size $j$, and hence
\[\sum_\lambda f_\lambda(t)=\sum_{j\geq 0}(-1)^j\chi(\Omega^j(t))=\binom{n}{k}\]
with the first sum now over all bounded partitions.
\end{proof}

\begin{example}[$t=0$]
    All partitions are $0$-core, and $f_\lambda(0)=1$ for every $\lambda$. The sum $\sum_\lambda f_\lambda(0)$ is the number of bounded partitions $\binom n k$, and the holomorphic Euler characteristic is $\sum (-1)^j\chi(\Omega^j(0))=\chi(\mathrm{Gr}(k,n))$.
\end{example}

\begin{example}[$t=1$]
    The only $1$-core partition is the empty partition $(0^k)$, and $f_{(0^k)}(1)=\binom{n}{k}$ by \Cref{ex: Macmahon}. Algebraically this follows from
\[h^0(\text{Gr}(k,n),\mathcal O(1))=\binom{n}{k}.\]
\end{example}

\subsection{Combinatorial interpretation}

Because the $t$-hook ratio $f_\lambda(t)$ is an integer-valued polynomial in $t$ (\Cref{rmk:integer-valued}), so is $\sum_\lambda f_\lambda(t)$ with the sum over all bounded partitions $\lambda$. By \Cref{thm: euler characteristic hook length}, this polynomial is constant.

Suppose $(a_1,b_1),\dots,(a_N,b_N)$ is an arbitrary linear ordering on the boxes of the $k\times(n-k)$ bounding rectangle. By expanding $f_\lambda(t)$, we see that the coefficient of $t^d$ is given by
\[[t^d]f_\lambda(t)=\sum_{1\leq\ell_1<\cdots<\ell_d\leq N}\frac{(-1)^d}{h_\lambda(a_{\ell_1},b_{\ell_1})\cdots h_\lambda(a_{\ell_d},b_{\ell_d})},\]
where $h_\lambda(a,b)=\lambda_a+\lambda_b^\T-a-b+1$ allows for both positive and negative hook lengths as in \Cref{sec:partitions}. The following statement (plus the calculation for $t=0$) is equivalent to \Cref{thm: euler characteristic hook length}.

\begin{corollary}
For every $d\geq 1$,
\[\sum_\lambda\sum_{1\leq\ell_1<\cdots<\ell_d\leq N}\frac{1}{h_\lambda(a_{\ell_1},b_{\ell_1})\cdots h_\lambda(a_{\ell_d},b_{\ell_d})}=0,\]
where the first sum is over all $(k,n-k)$-bounded partitions.
\end{corollary}

The authors are unaware of a direct combinatorial proof of this fact. The case where $d$ is odd follows from the duality $\lambda\mapsto\lambda^\C$, since any $d$-fold product of hook lengths in $\lambda$ can be exchanged for their negatives in $\lambda^\C$, which introduces a negative overall.

\begin{example}
For $k=2$ and $n=4$ we have the partitions:

\begin{center}
\begin{tabular}{r | c c c c c c }
\raisebox{1.5em}{$\lambda$} & $\begin{tikzpicture}[scale = 0.6]
\draw [black!50!white] (0,0) grid (2,2);
\draw [thick] (0,0) grid (2,2);
\node at (0.5,1.5) {$3$};
\node at (1.5,1.5) {$2$};
\node at (0.5,0.5) {$2$};
\node at (1.5,0.5) {$1$};
\end{tikzpicture}$ & $\begin{tikzpicture}[scale = 0.6]
\draw [black!50!white] (0,0) grid (2,2);
\draw [thick] (0,0) grid (1,1);
\draw [thick] (0,1) grid (2,2);
\node at (0.5,1.5) {$3$};
\node at (1.5,1.5) {$1$};
\node at (0.5,0.5) {$1$};
\node at (1.5,0.5) {$-1$};
\end{tikzpicture}$ & $\begin{tikzpicture}[scale = 0.6]
\draw [black!50!white] (0,0) grid (2,2);
\draw [thick] (0,0) grid (1,2);
\node at (0.5,1.5) {$2$};
\node at (1.5,1.5) {$-1$};
\node at (0.5,0.5) {$1$};
\node at (1.5,0.5) {$-2$};
\end{tikzpicture}$ & $\begin{tikzpicture}[scale = 0.6]
\draw [black!50!white] (0,0) grid (2,2);
\draw [thick] (0,1) grid (2,2);
\node at (0.5,1.5) {$2$};
\node at (1.5,1.5) {$1$};
\node at (0.5,0.5) {$-1$};
\node at (1.5,0.5) {$-2$};
\end{tikzpicture}$ & $\begin{tikzpicture}[scale = 0.6]
\draw [black!50!white] (0,0) grid (2,2);
\draw [thick] (0,1) grid (1,2);
\node at (0.5,1.5) {$1$};
\node at (1.5,1.5) {$-1$};
\node at (0.5,0.5) {$-1$};
\node at (1.5,0.5) {$-3$};
\end{tikzpicture}$ & $\begin{tikzpicture}[scale = 0.6]
\draw [black!50!white] (0,0) grid (2,2);
\node at (0.5,1.5) {$-1$};
\node at (1.5,1.5) {$-2$};
\node at (0.5,0.5) {$-2$};
\node at (1.5,0.5) {$-3$};
\end{tikzpicture}$ \\

$[t^2]f_\lambda(t)$ & $\dfrac{23}{12}$ & $-\dfrac{2}{3}$ & $-\dfrac{5}{4}$ & $-\dfrac{5}{4}$ & $-\dfrac{2}{3}$ & $\dfrac{23}{12}$
\end{tabular}
\end{center}
and the sum of the values $[t^2]f_\lambda(t)$ is zero.
\end{example}

\section{A \texorpdfstring{$q$}{q}-analogue}
\label{sec:q}
In this section, we compute the $q$-dimension (as a $\mathrm{GL}_n$ representation) of the cohomology $H^i(\Omega_X^j(t))$, discuss some examples, and conjecture a $q$-analogue of \Cref{thm: euler characteristic hook length}. We recall the notation and definitions from \Cref{sec:rep-theory}.

\subsection{Setup}

Following \cite{Stem94}, the $q$-dimension of a (finite dimensional) rational representation $V$ of $\mathrm{GL}_n$ is defined as the Laurent polynomial
\[\dim_q(V):=\sum_{\mu\in\Z^n}\dim(V[\mu])q^{\langle\varrho,\mu\rangle}\in\Z[q^{\pm 1/2}].\]
We make special note of the use of $\varrho$ as the Weyl vector here, rather than the shifted version $\varrho'$ that has been employed so far. The benefit of this choice is the identity
\[\dim_q(V)=\dim_{q^{-1}}(V)=\dim_q(V^\vee).\]
This holds by recalling that $V[\mu]$, $V[\mathrm{rev}(\mu)]$ and $V^\vee[-\mu]$ have the same dimension for every $\mu\in\Z^n$, and computing that $\langle \mathrm{rev}(\mu),\varrho\rangle=-\langle\mu,\varrho\rangle=\langle-\mu,\varrho\rangle$.

\begin{definition}
For a bounded $t$-core partition $\lambda$, set $[d_\lambda(t)]_q:=\dim_q L_{\beta_\lambda(t)}^\vee$.
\end{definition}

The following is immediate by taking the $q$-dimension of both sides of the equality in \Cref{thm:bwb}.

\begin{corollary}
For integers $(k,n,j,t,i)$ with $1\leq k< n$ and $j,i\geq 0$:
\[[h^{j,i}(t)]_q:=\dim_q(H^i(\mathrm{Gr}(k,n),\Omega^j(t)))=\sum_\lambda [d_\lambda(t)]_q\]
with the sum over all Snow partitions $\lambda$ with parameters $(k,n,j,t,i)$.
\end{corollary}

We observe that $[d_\lambda(t)]_q$ and $[h^{j,i}(t)]_q$ are always polynomials in $\Z[q^{\pm 1/2}]$ which are invariant under $q\mapsto q^{-1}$.

\subsection{Calculation of the \texorpdfstring{$q$}{q}-dimension}

We first calculate $[d_\lambda(t)]_q$ in terms of the BWB weight $\beta_\lambda(t)\in\Lambda_n$ (see \Cref{def:beta}) using the $q$-analogue of the Weyl dimension formula. For every integer $m$, define the (symmetrized) $q$-analogue
\[[m]_q:=\frac{q^{m/2}-q^{-m/2}}{q^{1/2}-q^{-1/2}}=q^{-(m-1)/2}\frac{1-q^m}{1-q}\in\Z[q^{\pm 1/2}]\]
which satisfies $[m]_q=[m]_{q^{-1}}$ and $[-m]_q=-[m]_q$.

\begin{theorem}\label{lem:q-weyl}
For a Snow partition $\lambda$,
\[[d_\lambda(t)]_q=\prod_{1\leq r<s\leq n}\frac{[\beta_\lambda(t)_r-\beta_\lambda(t)_s+s-r]_q}{[s-r]_q}.\]
\end{theorem}
\begin{proof}
We have $[d_\lambda(t)]_q=\dim_q L_{\beta_\lambda(t)}^\vee=\dim_q L_{\beta_\lambda(t)}$. By \cite[2.7]{Stem94}, for any $\beta\in\Lambda_n$ we have
\[\dim_q L_\beta=q^{-\langle \beta,\varrho\rangle}\prod_{\phi\in\Phi^+}\frac{1-q^{\langle \beta+\varrho,\phi\rangle}}{1-q^{\langle\varrho,\phi\rangle}}=q^{-\langle \beta,\varrho\rangle}\prod_{1\leq r<s\leq n}\frac{1-q^{\beta_r-\beta_s+s-r}}{1-q^{s-r}}.\]
Converting each term in the product to the symmetrized $q$-analogue gives
\[\dim_q L_\beta=q^z\prod_{1\leq r<s\leq n}\frac{[\beta_r-\beta_s+s-r]_q}{[s-r]_q}\]
for some shift $z$. Because the left-hand side is invariant under $q\mapsto q^{-1}$, as is every term of the form $[m]_q$, we must have $z=0$.
\end{proof}
\begin{remark}
In the remaining proofs in this section, we will repeat this trick of denoting by $q^z$ some undetermined $q$-shift, which we reason to be $q^0$ by showing that $[d_\lambda(t)]_q$ is equal to the product of $q^z$ and terms of the form $[m]_q$.
\end{remark}

In the case of a polynomial representation $L_\gamma$ for $\gamma\in\Lambda_n^+$, we can count tableaux of shape $\gamma$ by their size. Recall that $L_{\beta_\lambda(t)}^\vee$ was either isomorphic or dual to $L_{\gamma_\lambda(t)}$ depending on the sign of $t$ (see \Cref{def:gamma} and its preceding discussion).

\begin{theorem}\label{thm:q-dimension-gamma}
For a Snow partition $\lambda$, set $\gamma:=\gamma_\lambda(t)$. Then
\[[d_\lambda(t)]_q=q^{-(n+1)(n-k)|t|/2}\sum_{T\in\SSYT(\gamma;n)}q^{|T|}=\prod_{(a,b)\in\gamma}\frac{[n+b-a]_q}{[h_\gamma(a,b)]_q}.\]
\end{theorem}
\begin{proof}
If $t\geq 0$, then $L_{\beta_\lambda(t)}^\vee\cong L_\gamma$, so $[d_\lambda(t)]_q=\dim_q L_\gamma$. If $t< 0$, then $L_{\beta_\lambda(t)}^\vee$ is dual to $L_\gamma$ so $[d_\lambda(t)]_q=\dim_{q^{-1}}L_\gamma=\dim_q L_\gamma$. In either case it remains to calculate the $q$-dimension of $L_\gamma$.

For any $\gamma\in\Lambda_n^+$, we have $\dim L_\gamma[\mu]=\#\SSYT(\gamma;\mu)$ for every $\mu\in\Z^n$ and so
\begin{align*}\dim_q(L_\gamma)&=\sum_{\mu\in\Z^n}\#\SSYT(\gamma;\mu)q^{\langle \mu,\varrho\rangle}\\
&=s_\gamma(q^{(n-1)/2},q^{(n-3)/2},\dots,q^{-(n-1)/2})\\
&=q^{-(n-1)|\gamma|/2}s_\gamma(1,q,\dots,q^{n-1}),\end{align*}
where we use the fact that $s_\gamma(x_1,\dots,x_n)$ is a symmetric polynomial of homogeneous degree $|\gamma|$. Reintroducing a factor of $q$ in each variable gives
\[q^{-(n+1)|\gamma|/2}s_\gamma(q,q^2,\dots,q^n)=q^{-(n+1)|\gamma|/2}\sum_{T\in\SSYT(\gamma;n)}q^{|T|}.\]
Recall that $|\gamma_\lambda(t)|=(n-k)|t|$ by \Cref{thm:gamma}, which gives the first equality.

For the second, \cite[\S3 Example 1]{Mac86} gives
\[s_\gamma(1,q,\dots,q^{n-1})=q^{z_1}\prod_{(a,b)\in\gamma}\frac{1-q^{n+b-a}}{1-q^{h_\gamma(a,b)}}\]
for some $z_1$, and so after applying the relevant $q$-shift we get
\[[d_\lambda(t)]_q=q^{z_2}\prod_{(a,b)\in\gamma}\frac{[n+b-a]_q}{[h_\gamma(a,b)]_q}\]
for some $z_2$. But by $q\mapsto q^{-1}$ symmetry, we must have $z_2=0$.
\end{proof}

Finally, we prove $[d_\lambda(t)]_q$ can also be calculated by the expected $q$-analogue of the $t$-hook ratio. 

\begin{definition}
For a bounded partition $\lambda$ and an integer $t$, set
\[[f_\lambda(t)]_q=\prod_{a=1}^k\prod_{b=1}^{n-k}\frac{[h_\lambda(a,b)-t]_q}{[h_\lambda(a,b)]_q}.\]
\end{definition}
\begin{remark}
Using $[m]_q=q^{-(m-1)/2}(1-q^m)/(1-q)$ we can rewrite the $q$-analogue of the $t$-hook ratio as
\[[f_\lambda(t)]_q=q^{-Nt/2}\prod_{a=1}^k\prod_{b=1}^{n-k}\frac{1-q^{h_\lambda(a,b)-t}}{1-q^{h_\lambda(a,b)}}.\]
\end{remark}

We start with a $q$-analogue of \Cref{lem:skew-howe-dimension}.

\begin{lemma}
For a bounded partition $\lambda$:
\[\dim_q L_\lambda\cdot \dim_q L_{\lambda^\T}=(-1)^{j-N}\prod_{a=1}^k\prod_{b=1}^{n-k}\frac{[a+b-1]_q}{[h_\lambda(a,b)]_q}\]
\end{lemma}
\begin{proof}
We follow the proof of \cite[Theorem 1.1]{Gre18}, but apply the $q$-analogue. There is a bijection $c:\mathrm{SSYT}(\lambda;k)\to\mathrm{SSYT}(\lambda^\C;k)$ since every column must contain every entry between $1$ and $k$. Moreover, $|c(T)|=(n-k)(1+\cdots+k)-|T|$. In particular, we obtain that
\begin{align*}
\dim_q L_\lambda&=q^{z_1}s_\lambda(1,q,\dots,q^{k-1})\\
&=q^{z_2}s_{\lambda^\C}(1,q^{-1},\dots,q^{-(k-1)})\\
&=q^{z_3}s_{\lambda^\C}(1,q,\dots,q^{k-1})\\
&=q^{z_4}\dim_q L_{\lambda^\C}
\end{align*}
and by $q\mapsto q^{-1}$ symmetry we have $z_4=0$. By applying the calculation in the proof of \Cref{thm:q-dimension-gamma} to $\lambda^{\T\C}$ we have
\[\dim_q L_{\lambda^\T}=\dim_q L_{\lambda^{\T\C}}=\prod_{(a,b)\in\lambda^{\T\C}}\frac{[n-k+b-a]_q}{[h_{\lambda^{\T\C}}(a,b)]_q}.\]
Reindexing the numerator by associating the boxes of $\lambda^{\T\C}$ with the boxes of the skew partition $(n-k)^k/\lambda$ under the map $(a,b)\mapsto (k-b+1,n-k-a+1)$ turns $n-k+b-a$ into $k+b-a$ and gives
\[\dim_q L_{\lambda^\T}=\frac{\prod_{(a,b)\in (n-k)^k/\lambda}[k+b-a]_q}{\prod_{(a,b)\in\lambda^\C}[h_{\lambda^\C}(a,b)]_q}.\]
Hence,
\[\dim_q L_\lambda\cdot \dim_q L_{\lambda^\T}=\frac{\prod_{a=1}^k\prod_{b=1}^{n-k}[k+b-a]_q}{\prod_{(a,b)\in\lambda}[h_\lambda(a,b)]_q\prod_{(a,b)\in\lambda^\C}[h_{\lambda^\C}(a,b)]_q}.\]
The terms in the numerator become $[a+b-1]_q$ after reindexing $a$ for $k-a+1$. In the denominator, the hook lengths in $\lambda^\C$ correspond to the negative hook lengths of $\lambda$, so the denominator is $(-1)^{N-j}\prod_{a=1}^k\prod_{b=1}^{n-k}[h_\lambda(a,b)]_q$ as required.
\end{proof}

\begin{theorem}\label{thm:q-dimension-hook}
For a Snow partition $\lambda$,
\[[d_\lambda(t)]_q=(-1)^{i+j}[f_\lambda(t)]_q.\]
\end{theorem}
\begin{proof}
We follow the proof of \Cref{thm:dimension-formula}. Since $\beta_\lambda(t)=\sigma(\alpha_\lambda(t)+\varrho)-\varrho$ where $\sigma$ is a permutation of length $i$, we have
\[[d_\lambda(t)]_q=(-1)^i\prod_{1\leq r<s\leq n}\frac{[\alpha_\lambda(t)_r-\alpha_\lambda(t)_s+s-r]_q}{[s-r]_q}.\]
We split this product into three terms. When $r,s\leq k$ we get $\dim_q L_{\lambda}$ and when $r,s>k$ we get $\dim_q L_{\lambda^\T}$. The final term is
\[(-1)^N\prod_{a=1}^k\prod_{b=1}^{n-k}\frac{[h_\lambda(a,b)-t]_q}{[a+b-1]_q}\]
so multiplying these terms together and applying the above lemma gives the result.
\end{proof}

\subsection{Examples}

We begin with a generic example to showcase various methods for calculating $[d_\lambda(t)]_q$, before taking the $q$-analogue of some examples in \Cref{sec:examples}.
\begin{example}
Let $\lambda=(2,1,1)$, which is the only Snow partition with parameters $k=3$, $n=5$, $j=4$, $t=3$ and $i=1$.
\[\begin{tikzpicture}[scale = 0.6]
\draw [black!50!white] (0,0) grid (2,3);
\draw [fill = intrColor] (0,2) -- (1,2) -- (1,3) -- (0,3) -- cycle;
\draw [thick] (0,0) grid (1,2);
\draw [thick] (0,2) grid (2,3);
\end{tikzpicture}\]
Then $\dim_q H^1(\mathrm{Gr}(3,5),\Omega^4(3))=[d_\lambda(3)]_q$. We first calculate via $\gamma_\lambda(t)$. Because $\partial_\lambda(3)=(1,1,1)$ and $3-\partial_{\lambda^\T}(3)=(1,2)$ we have $\gamma_\lambda(t)=(2,1^4)$. Hence, we want the weights of all tableaux in $\SSYT((2,1^4);5)$.
\[\begin{tikzpicture}[scale = 0.4]
\foreach \i in {0,1,2,3,4} {
\draw (3*\i,0) grid (3*\i+2,-1);
\draw (3*\i,-1) grid (3*\i+1,-5);
	\foreach \j in {1,2,3,4,5} {
	\node at (3*\i+0.5,0.5-\j) {$\j$};}
}
\node at (-2,-2.5) {$T$};
\node at (-2,-6) {$|T|$};

\node at (1.5,-0.5) {$1$};
\node at (1,-6) {$16$};

\node at (4.5,-0.5) {$2$};
\node at (4,-6) {$17$};

\node at (7.5,-0.5) {$3$};
\node at (7,-6) {$18$};

\node at (10.5,-0.5) {$4$};
\node at (10,-6) {$19$};

\node at (13.5,-0.5) {$5$};
\node at (13,-6) {$20$};

\draw (-1,0) -- (-1,-6.5);
\end{tikzpicture}\]
Since $-(n+1)(n-k)|t|/2=-18$ we have \[[d_\lambda(t)]_q=q^{-18}(q^{16}+q^{17}+q^{18}+q^{19}+q^{20})=q^{-2}+q^{-1}+1+q+q^2.\]

Notice that the shift symmetrizes $\sum_T q^{|T|}$ about $q\mapsto q^{-1}$. Alternatively, we could have calculated $[d_\lambda(t)]_q$ using the second formula in \Cref{thm:q-dimension-gamma}. Instead, we use the $t$-hook ratio for $\lambda$ and \Cref{thm:q-dimension-hook} to get
\[[d_\lambda(t)]_q=(-1)^{1+4}[f_\lambda(3)]_q=-\frac{[1]_q}{[4]_q}\cdot\frac{[-1]_q}{[2]_q}\cdot\frac{[-2]_q}{[1]_q}\cdot\frac{[-2]_q}{[1]_q}\cdot\frac{[-4]_q}{[-1]_q}\cdot\frac{[-5]_q}{[-2]_q}\]
which after cancelling terms is $-[-5]_q=[5]_q=q^{-2}+q^{-1}+1+q+q^2$.
\end{example}

\begin{example}[$j=0$]
Following \Cref{ex: Macmahon}, the only Snow partition with $j=0$ is $\lambda=(0^k)$, and for every $t\geq 0$ we obtain
\begin{align*}[d_\lambda(t)]_q&=\prod_{a=1}^k\prod_{b=1}^{n-k}\frac{[a+b+t-1]_q}{[a+b-1]_q}\\
&=q^{-Nt/2}\prod_{a=1}^k\prod_{b=1}^{n-k}\frac{1-q^{a+b+t-1}}{1-q^{a+b-1}}\\
&=q^{-Nt/2}\sum_P q^{|P|}
\end{align*}
where the final sum is over all $(k,n-k,t)$-bounded plane partitions (see \cite{Mac86}).
\end{example}

\begin{definition}
For integers $1\leq k< n$, define the (symmetrized) $q$-analogue of the binomial coefficient
\[\begin{bmatrix}n\\k\end{bmatrix}_q=\frac{[n]_q\cdots [1]_q}{[k]_q\cdots [1]_q[n-k]_q\cdots [1]_q}\]
\end{definition}

\begin{example}[$t=2$]
    Following \Cref{ex:t=2}, assume that $k\leq n/2$, and set $\Delta_m:=(m,\dots,2,1)$ for $0\leq m\leq k$. We have
\[\gamma_{\Delta_m}(2)=(2^{n-k-m},1^{2m},0^{k-m}).\]
Applying \Cref{thm:q-dimension-gamma} to this shape we get
\[[h^{m(m+1)/2,m(m-1)/2}(2)]_q=\frac{[2m+1]_q}{[n+1]_q}\begin{bmatrix}n+1\\k+m+1\end{bmatrix}_q\begin{bmatrix}n+1\\n-k+m+1\end{bmatrix}_q\]
for every $0\leq m\leq \min(k,n-k)$, and $[h^{j,i}(2)]_q=0$ otherwise.
\end{example}

Finally, we conjecture a $q$-analogue for \Cref{thm: euler characteristic hook length}, which we hope can lead to a combinatorial proof of that result.

\begin{conjecture}
    \label{con:q-euler-characteristic}
Fix $1\leq k< n$. Then for every integer $t$,
\[\sum_{\lambda}[f_{\lambda}(t)]_q=\begin{bmatrix}n\\k\end{bmatrix}_{q^t}\]
where the sum is over all $(k,n-k)$-bounded partitions $\lambda$.
\end{conjecture}

\printbibliography

\end{document}